\documentclass[11pt]{article}

\usepackage[a4paper,margin=30mm]{geometry}
\usepackage{amsmath,amssymb,amsthm,mathtools}
\usepackage{mathrsfs}
\usepackage{microtype}
\usepackage[hidelinks]{hyperref}
\hypersetup{
  pdftitle={A Critical Quantitative Landis Estimate for the One-Dimensional Quarter-Laplacian},
  pdfauthor={Adham Gudaimat},
  pdfsubject={Critical quantitative Landis estimate, rough Grushin--Robin Carleman estimate, and quarter-Laplacian},
  pdfkeywords={fractional Landis conjecture, quarter-Laplacian, Grushin operator, Robin boundary feedback, quantitative unique continuation}
}

\newtheorem{theorem}{Theorem}[section]
\newtheorem{proposition}[theorem]{Proposition}
\newtheorem{lemma}[theorem]{Lemma}
\newtheorem{corollary}[theorem]{Corollary}
\theoremstyle{definition}
\newtheorem{definition}[theorem]{Definition}
\theoremstyle{remark}
\newtheorem{remark}[theorem]{Remark}

\newcommand{\R}{\mathbb{R}}
\newcommand{\HH}{\mathcal{H}}
\newcommand{\BB}{\mathcal{B}}
\newcommand{\dd}{\,\mathrm{d}}

\title{\textbf{A Critical Quantitative Landis Estimate\\
for the One-Dimensional Quarter-Laplacian}}

\author{
Adham Gudaimat\\
\small Ministry of Education and Higher Education, Ramallah, Palestine\\
\small \texttt{adhamgudaimat@nhebron.edu.ps}\\
\small ORCID: 0009-0001-3649-6701
}

\date{}

\begin{document}
\maketitle

\begin{abstract}
We establish a quantitative Landis estimate for the one-dimensional
fractional Schr\"odinger equation
\[
(-\Delta)^{1/4}u+V(x)u=0
\qquad\text{in }\mathbb R
\]
with a real-valued bounded potential. If
\(\|V\|_{L^\infty}\le1\), \(\|u\|_{L^\infty}\le C_0\), and
\(\|u\|_{L^2(-1,1)}\ge1\), then
\[
\inf_{|x_0|=R}\|u\|_{L^\infty(x_0-1,x_0+1)}
\ge \exp(-C R\log R)
\]
for all sufficiently large \(R\). After the Caffarelli--Silvestre extension
and the substitution \(y=z^2/2\), the equation becomes a Grushin equation
with a weak Robin condition on the degeneracy line. The corresponding
angular operator has the arithmetic spectrum
\(\kappa_n=2n+\tfrac12\) after half-density conjugation. The central
spectral estimate is
\[
\sup_{\xi\in\mathbb R}
\bigl\|C\bigl((\tau+i\xi)^2-L_0\bigr)^{-1}C^*\bigr\|
\le C\tau^{-1/2}
\]
for parameters separated from the angular lattice. It yields a
linear-weight Carleman estimate for measurable Robin feedback with
absorption threshold \(\tau\ge C(1+\|V\|_\infty^2)\). Quantitative inward
propagation, fixed-scale Grushin-ball propagation, and an
interior-cylinder interpolation estimate then transfer bulk non-vanishing
to the boundary. The Landis rescaling converts the local potential
dependence \(C\|q\|_\infty^2\) into the global rate \(CR\log R\).
\end{abstract}


\section{Introduction}

The classical Landis problem asks how rapidly a nontrivial bounded
solution of a Schr\"odinger equation can decay at infinity. Meshkov
\cite{Meshkov} constructed complex-valued bounded potentials admitting
solutions with decay of order \(\exp(-c|x|^{4/3})\), while the quantitative
unique-continuation estimate of Bourgain and Kenig \cite{BK} supplied the
corresponding large-scale lower-bound mechanism. In the real-valued planar
setting, Logunov, Malinnikova, Nadirashvili, and Nazarov \cite{LMNN}
proved the weak Landis conjecture with the decay threshold
\(\exp(-C|x|\sqrt{\log |x|})\).

We consider the one-dimensional fractional Schr\"odinger equation
\begin{equation}\label{eq:frac-sch}
(-\Delta)^{1/4}u+V(x)u=0
\qquad\text{in }\mathbb R,
\end{equation}
with \(V\in L^\infty(\mathbb R;\mathbb R)\). For bounded,
non-differentiable potentials, the quantitative fractional Landis estimates
of R\"uland and Wang \cite{RW} and their variable-coefficient extension by
Kow \cite{KowGeneral} apply when \(s\in(1/4,1)\); neither result includes
the endpoint \(s=1/4\). The rough-potential theory in
\cite{RulandRough,RulandQuant} gives qualitative unique continuation and
local quantitative estimates, but not the solution-independent critical
dependence required by the Landis rescaling. To the best of our knowledge,
no previous result establishes a quantitative Landis estimate at
\(s=1/4\) for one-dimensional real-valued bounded potentials under the
hypotheses stated below. The purpose of this paper is to prove that endpoint
estimate by exploiting the explicit angular spectrum of the associated
Grushin geometry.

The quarter-Laplacian extension \cite{CS}, followed by the change of
variables \(y=z^2/2\), transforms \eqref{eq:frac-sch} into
\[
W_{zz}+z^2W_{xx}=0
\]
in the upper half-plane, with a Robin condition supported on the degeneracy
line \(z=0\). The Robin condition is used throughout in its weak conormal
form, so no pointwise normal derivative of a rough solution is required.
After passing to homogeneous Grushin polar coordinates and logarithmic
radius, the problem becomes a second-order equation on a finite angular
interval with measurable endpoint feedback. The angular operator has an
explicit Gegenbauer basis with eigenvalues \(2n(2n+1)\), and the normalized
eigenfunctions satisfy
\[
|\Phi_n(0)|^2=|\Phi_n(\pi)|^2\asymp(n+1)^{1/2}.
\]
This endpoint growth produces a loss of \(3/4\) of an angular derivative,
a Robin spectral lower scale of order \(-\mu^4\), and a logarithmic radial
scale of order \(\mu^2\).

Two quantitative mechanisms are developed. First, low angular frequencies
are controlled by an exponential-polynomial observability estimate that is
stable under measurable endpoint feedback. High frequencies are represented
by a Dirichlet Green kernel on a larger logarithmic cylinder. The interior
margin gives exponential decay of the homogeneous high-frequency component,
whereas the forced endpoint operator has norm \(O(\mu N^{-1/2})\). The
choice \(N\asymp1+\mu^2\) prevents cancellation between the low- and
high-frequency endpoint traces. In particular,
Theorem~\ref{thm:full-cylinder} yields constants \(d_*>1\), \(C,c>0\)
such that
\[
\|w(0)\|_{\mathcal H}^2
\le
\exp(CN)\|Cw\|_{L^2(J_{d_*})}^2
+Ce^{-cN}\mathcal M_{d_*}(w)^2,
\qquad N\ge C(1+\mu^2).
\]
Optimization in \(N\) gives an interpolation cost
\(\exp(C(1+\mu^2))\). Here ``observability'' is used in the
spectral-interpolation sense; the logarithmic radius is an auxiliary
elliptic variable rather than a physical evolution time.

Second, the half-density conjugation exposes the arithmetic lattice
\(\kappa_n=2n+\tfrac12\). A discrete Cauchy-transform estimate on the two
parity sublattices gives the critical boundary-resolvent bound
\[
\sup_{\xi\in\mathbb R}
\left\|
C\bigl((\tau+i\xi)^2-L_0\bigr)^{-1}C^*
\right\|
\le C\tau^{-1/2}
\]
for admissible \(\tau\). The measurable Robin term is therefore absorbable
when \(\tau\ge C(1+\mu^2)\). The resulting linear-weight Carleman estimate
gives quantitative inward propagation in logarithmic cylinders. A
fixed-scale boundary seed is propagated by explicit overlapping Grushin
balls, and the annular interpolation estimate transfers the resulting bulk
lower bound back to the boundary trace.

The contributions are therefore: a distributional formulation for bounded
solutions, including local \(H^{1/2}\)-regularity and a weak
quarter-Laplacian--Grushin--Robin reduction; the logarithmic-cylinder
formulation with explicit endpoint coefficients; the angular spectrum and
endpoint normalization; the critical \(3/4\)-trace inequality and the
\(-\mu^4\) Robin lower scale; high-frequency coercivity above
\(N\asymp1+\mu^2\); finite-dimensional observability stable under measurable
endpoint feedback; interior-cylinder interpolation with cost
\(\exp(C(1+\mu^2))\); an energy-class Fourier--resolvent representation and
the boundary-resolvent estimate without logarithmic loss; and a centered
quantitative vanishing estimate, propagated through a finite chain of
Grushin balls, which yields the global lower bound \(\exp(-CR\log R)\).

The argument differs from existing Carleman and frequency methods for
Baouendi--Grushin equations in \cite{GS,GV,BG16,BGM}. A quantitative Landis
estimate for a bulk degenerate Grushin equation was obtained by Chen, Liu and
Yang \cite{ChenLiuYang}. Recent work also includes higher-step Grushin
Carleman estimates \cite{DeBieLian}, quantitative continuation for
Baouendi--Grushin type operators \cite{SunLiu26}, and strong unique
continuation for the Baouendi operator with zero-order perturbations
\cite{BG26}; spectral inequalities on product cylinders are developed in
\cite{AS}. These results treat interior zero-order perturbations of
degenerate operators. Here the coefficient is transferred to the degeneracy
set as merely measurable Robin feedback, so the decisive estimate is instead
a boundary resolvent for the endpoint trace. Kow and Wang \cite{KW} use the
Caffarelli--Silvestre extension for a Landis problem involving the
half-Laplacian, whose harmonic-extension geometry does not contain the
critical Grushin--Robin mechanism considered here. Quantitative Robin
continuation for uniformly elliptic problems is treated in
\cite{Li,LiWang,DZ}; those results do not directly cover the present
degenerate endpoint-feedback setting.

\section{Main result and scaling reduction}
\label{sec:landis-reduction}

We first state the quantitative endpoint problem in the normalization used
throughout this paper.

\begin{theorem}[Critical quantitative Landis estimate]
\label{thm:quant-landis}
Let $V\in L^\infty(\mathbb R;\mathbb R)$ and
$u\in L^\infty(\mathbb R;\mathbb R)$. Assume that
\[
(-\Delta)^{1/4}u+Vu=0\quad\text{in }\mathcal S'(\mathbb R),
\qquad
\|V\|_{L^\infty(\mathbb R)}\le1,
\qquad
\|u\|_{L^\infty(\mathbb R)}\le C_0,
\]
together with
\[
\|u\|_{L^2(-1,1)}\ge1.
\]
Then there exist constants $C,R_0>0$, depending only on $C_0$, such
that for every $R\ge R_0$,
\begin{equation}\label{eq:global-landis-target}
\inf_{|x_0|=R}\|u\|_{L^\infty(x_0-1,x_0+1)}
\ge \exp(-CR\log R).
\end{equation}
\end{theorem}

The following local statement is the exact quantitative input required by
the classical scaling mechanism.

\begin{definition}[Centered endpoint vanishing estimate]
\label{def:centered-vanishing}
We say that the centered endpoint estimate holds if there exist
\(c,C,r_0>0\), depending only on \(C_0\), with the following property. For every real-valued $v\in L^\infty(\mathbb R)$ satisfying
\[
(-\Delta)^{1/4}v+qv=0\quad\text{in }\mathcal S'(\mathbb R),
\]
and
\[
\|q\|_\infty\le M,\qquad
\|v\|_\infty\le C_0,\qquad
\|v\|_{L^2(-2,2)}\ge K\in(0,1],
\]
one has, for $0<r<r_0$,
\begin{equation}\label{eq:centered-vanishing}
\|v\|_{L^\infty(-r,r)}
\ge c\,r^{C(1+M^2+|\log K|)}.
\end{equation}
\end{definition}

\begin{theorem}[Sharp reduction to centered propagation]
\label{thm:scaling-reduction}
If the centered endpoint estimate \eqref{eq:centered-vanishing} holds, then
Theorem~\ref{thm:quant-landis} follows.
\end{theorem}

\begin{proof}
Fix $x_0$ with $|x_0|=R$ and define
\[
v_R(x)=u(x_0+Rx),
\qquad
q_R(x)=R^{1/2}V(x_0+Rx).
\]
The homogeneity of the quarter-Laplacian gives
\[
(-\Delta)^{1/4}v_R+q_Rv_R=0,
\qquad
\|q_R\|_\infty\le R^{1/2},
\qquad
\|v_R\|_\infty\le C_0.
\]
For $R\ge2$, the interval $x_0+R(-2,2)$ contains $(-1,1)$. Hence
\[
\|v_R\|_{L^2(-2,2)}^2
=R^{-1}\|u\|_{L^2(x_0+R(-2,2))}^2
\ge R^{-1},
\]
so the local normalization is valid with $K=R^{-1/2}$. Apply
\eqref{eq:centered-vanishing} with
\[
M=R^{1/2},\qquad K=R^{-1/2},\qquad r=R^{-1}.
\]
We obtain
\[
\begin{aligned}
\|u\|_{L^\infty(x_0-1,x_0+1)}
&=\|v_R\|_{L^\infty(-R^{-1},R^{-1})}\\
&\ge cR^{-C(1+R+\frac12\log R)}\\
&\ge \exp(-C_1R\log R)
\end{aligned}
\]
for all sufficiently large $R$. Taking the infimum over $|x_0|=R$ proves
\eqref{eq:global-landis-target}.
\end{proof}

\begin{remark}[The critical power of the potential]
Under the Landis rescaling, $M=\|q_R\|_\infty$ has size $R^{1/2}$.
Consequently, the local exponent $M^2$ becomes $R$, and evaluation at the
radius $R^{-1}$ produces $R\log R$. A larger power of $M$ would yield a
weaker global rate. The power $M^2$ agrees with the endpoint scaling
$M^{1/(2s)}$ at $s=1/4$.
\end{remark}

\section{Critical extension and Grushin reduction}

Let \(U(x,y)\) be the Caffarelli--Silvestre extension of \(u\). For
\(s=1/4\),
\begin{equation}\label{eq:extension}
\partial_y\!\left(y^{1/2}U_y\right)
+y^{1/2}U_{xx}=0,
\qquad y>0.
\end{equation}
Equivalently,
\[
U_{yy}+\frac{1}{2y}U_y+U_{xx}=0.
\]

Set
\[
y=\frac{z^2}{2},
\qquad
W(x,z)=U\!\left(x,\frac{z^2}{2}\right).
\]
A direct calculation gives
\[
U_y=\frac1zW_z,
\qquad
U_{yy}=\frac1{z^2}W_{zz}-\frac1{z^3}W_z,
\]
and hence
\begin{equation}\label{eq:grushin}
\BB W:=W_{zz}+z^2W_{xx}=0,
\qquad z>0.
\end{equation}

\begin{proposition}[Exact weak boundary reduction]\label{prop:reduction}
Let \(U\) be the Caffarelli--Silvestre extension of a bounded distributional
solution of \eqref{eq:frac-sch}, and set
\(W(x,z)=U(x,z^2/2)\). Then \(W(\cdot,0)=u\),
\(W_{zz}+z^2W_{xx}=0\) in \(z>0\), and the Robin condition is the weak
conormal identity
\begin{equation}\label{eq:weak-robin-reduction}
\int_{z>0}\left(W_z\Psi_z+z^2W_x\Psi_x\right)\,dx\,dz
=-\int_{\mathbb R}b(x)W(x,0)\Psi(x,0)\,dx
\end{equation}
for every \(\Psi\in C_c^\infty(\overline{\mathbb R^2_+})\), where
\[
b=\kappa_{1/4}V,
\qquad
\kappa_{1/4}=\frac{\sqrt2}{d_{1/4}}>0.
\]
Equivalently,
\[
-\partial_zW\big|_{z=0}=-bW(\cdot,0)
\]
in the conormal distributional sense (and in the local negative trace space
once the finite-energy realization is used). Conversely, if \(W\) is the
finite-energy Caffarelli--Silvestre extension of its trace and satisfies
\eqref{eq:weak-robin-reduction}, then the trace satisfies
\eqref{eq:frac-sch} distributionally.
\end{proposition}

\begin{proof}
For finite-energy data, the Dirichlet-to-Neumann identity reads
\[
(-\Delta)^{1/4}u
=-d_{1/4}\lim_{y\downarrow0}y^{1/2}U_y
=-\frac{d_{1/4}}{\sqrt2}\,\partial_zW\big|_{z=0}
\]
in the dual trace space. Testing the weighted extension equation and using
\((-\Delta)^{1/4}u=-Vu\) gives \eqref{eq:weak-robin-reduction} after the
change \(y=z^2/2\). The converse follows by reversing the same identity.
For merely bounded data, the truncation and local-energy passage establishing
this formula are given in Proposition~\ref{prop:extension-realization}.
\end{proof}

\section{Grushin polar coordinates and the angular operator}

Introduce
\[
z=\rho\sqrt{\sin\varphi},
\qquad
x=\frac{\rho^2}{2}\cos\varphi,
\qquad
\rho>0,\quad 0<\varphi<\pi.
\]
Then
\[
\rho=(z^4+4x^2)^{1/4}.
\]

\begin{proposition}[Polar decomposition]\label{prop:polar}
For every smooth function supported away from the origin,
\[
\BB
=
\sin\varphi\left[
\partial_{\rho\rho}
+\frac2\rho\partial_\rho
+\frac4{\rho^2}
\left(
\partial_{\varphi\varphi}
+\frac12\cot\varphi\,\partial_\varphi
\right)
\right].
\]
\end{proposition}

\begin{proof}
The inverse coordinate derivatives are
\[
\rho_z=\sin^{3/2}\varphi,
\qquad
\rho_x=\frac{\cos\varphi}{\rho},
\]
and
\[
\varphi_z=\frac{2\sqrt{\sin\varphi}\cos\varphi}{\rho},
\qquad
\varphi_x=-\frac{2\sin\varphi}{\rho^2}.
\]
Consequently,
\[
\partial_z
=
\sin^{3/2}\varphi\,\partial_\rho
+
\frac{2\sqrt{\sin\varphi}\cos\varphi}{\rho}\,
\partial_\varphi,
\]
and
\[
\partial_x
=
\frac{\cos\varphi}{\rho}\,\partial_\rho
-
\frac{2\sin\varphi}{\rho^2}\,\partial_\varphi.
\]
Substituting these identities into
\(\BB=\partial_z^2+z^2\partial_x^2\), using
\(z^2=\rho^2\sin\varphi\), and collecting the radial, angular,
and mixed terms gives the asserted formula; the mixed
\(\partial_{\rho\varphi}\)-terms cancel identically.
\end{proof}

Set
\[
\omega(\varphi)=(\sin\varphi)^{1/2},
\qquad
\HH=L^2((0,\pi),\omega(\varphi)\dd\varphi),
\]
and define the densely defined quadratic form
\[
\mathfrak a_0[f,g]
=
4\int_0^\pi
\omega(\varphi)f'(\varphi)\overline{g'(\varphi)}\dd\varphi.
\]
Its form domain is
\[
\mathcal V
=
\left\{
 f\in\HH:
 f\in H^1_{\mathrm{loc}}(0,\pi),\quad
 \int_0^\pi\omega|f'|^2\dd\varphi<\infty
\right\}.
\]
The form is closed, nonnegative, and defines a self-adjoint Friedrichs
realization, denoted by \(A_0\), of
\begin{equation}\label{eq:A0}
A_0f
=
-4\left(
 f''+\frac12\cot\varphi\,f'
\right)
=
-4\omega^{-1}(\omega f')'.
\end{equation}
The natural endpoint condition is the vanishing of the weighted flux
\(\omega f'\) at \(0\) and \(\pi\), whenever the latter is defined
classically.

\begin{proposition}[Angular spectrum and exact normalization]
\label{prop:spectrum}
Let \(C_n^{1/4}\) denote the Gegenbauer polynomial of degree \(n\) and
parameter \(1/4\). Define
\[
\mathfrak h_n
=
\frac{\pi\sqrt2\,\Gamma(n+\tfrac12)}
{\Gamma(n+1)(n+\tfrac14)\Gamma(\tfrac14)^2}
\]
and
\[
\Phi_n(\varphi)
=
\mathfrak h_n^{-1/2}C_n^{1/4}(\cos\varphi).
\]
Then \(\{\Phi_n\}_{n\ge0}\) is an orthonormal basis of \(\HH\), and
\[
A_0\Phi_n=2n(2n+1)\Phi_n.
\]
Consequently, if
\[
T=\left(A_0+\frac14I\right)^{1/2}-\frac12I,
\]
then
\[
T\Phi_n=2n\Phi_n.
\]
\end{proposition}

\begin{proof}
With \(s=\cos\varphi\), the unitary change of variables
\[
F(s)=f(\arccos s)
\]
maps \(\HH\) onto
\[
L^2\bigl((-1,1),(1-s^2)^{-1/4}\dd s\bigr).
\]
Moreover,
\[
f''+\frac12\cot\varphi\,f'
=
(1-s^2)F''-\frac32sF'.
\]
The Gegenbauer equation with parameter \(\lambda=1/4\) is
\[
(1-s^2)Y''-\frac32sY'
+n\left(n+\frac12\right)Y=0.
\]
Therefore
\[
A_0C_n^{1/4}(\cos\varphi)
=4n\left(n+\frac12\right)C_n^{1/4}(\cos\varphi)
=2n(2n+1)C_n^{1/4}(\cos\varphi).
\]
The classical Gegenbauer orthogonality formula \cite{DLMF,Szego} gives
\[
\int_{-1}^{1}
(1-s^2)^{-1/4}
\bigl(C_n^{1/4}(s)\bigr)^2\dd s
=
\mathfrak h_n,
\]
while distinct degrees are orthogonal. Completeness follows from the
completeness of the Gegenbauer system for the weight
\((1-s^2)^{-1/4}\). The assertion for \(T\) follows from
\[
2n(2n+1)+\frac14
=\left(2n+\frac12\right)^2.
\]
\end{proof}

\section{Endpoint growth and the critical trace inequality}

The parity identity for Gegenbauer polynomials gives
\[
\Phi_n(\pi)=(-1)^n\Phi_n(0).
\]

\begin{proposition}[Exact endpoint formula]\label{prop:endpoint-growth}
For every \(n\ge0\),
\begin{equation}\label{eq:endpoint-exact}
|\Phi_n(0)|^2
=|\Phi_n(\pi)|^2
=
\frac{\Gamma(\tfrac14)^2}{\pi^2\sqrt2}
\left(n+\frac14\right)
\frac{\Gamma(n+\tfrac12)}{\Gamma(n+1)}.
\end{equation}
In particular, there exist constants \(0<c<C<\infty\) such that
\begin{equation}\label{eq:endpoint-growth}
c(n+1)^{1/2}
\le
|\Phi_n(0)|^2+|\Phi_n(\pi)|^2
\le
C(n+1)^{1/2}
\end{equation}
for all \(n\ge0\).
\end{proposition}

\begin{proof}
The classical endpoint value is
\[
C_n^{\lambda}(1)
=
\frac{\Gamma(n+2\lambda)}
{\Gamma(2\lambda)\Gamma(n+1)}.
\]
For \(\lambda=1/4\),
\[
C_n^{1/4}(1)
=
\frac{\Gamma(n+\tfrac12)}
{\sqrt\pi\,\Gamma(n+1)}.
\]
Dividing its square by \(\mathfrak h_n\) yields
\eqref{eq:endpoint-exact}. Finally,
\[
\frac{\Gamma(n+\tfrac12)}{\Gamma(n+1)}
=n^{-1/2}\bigl(1+O(n^{-1})\bigr),
\]
and the finitely many small values of \(n\) may be absorbed into the
constants, proving \eqref{eq:endpoint-growth}.
\end{proof}

\begin{theorem}[Critical spectral trace inequality]\label{thm:trace}
Every \(f\in\mathcal V=D(A_0^{1/2})\) has well-defined endpoint traces and
satisfies
\begin{equation}\label{eq:trace}
|f(0)|^2+|f(\pi)|^2
\le
C\left(
\mathfrak a_0[f,f]^{3/4}
\|f\|_{\HH}^{1/2}
+
\|f\|_{\HH}^2
\right).
\end{equation}
Equivalently, for every \(\varepsilon\in(0,1]\),
\begin{equation}\label{eq:trace-epsilon}
|f(0)|^2+|f(\pi)|^2
\le
C\left(
\varepsilon\mathfrak a_0[f,f]
+
\varepsilon^{-3}\|f\|_{\HH}^2
\right).
\end{equation}
\end{theorem}

\begin{proof}
It suffices first to consider a finite expansion
\[
f=\sum_{n=0}^{K}a_n\Phi_n.
\]
Write
\[
L=\|f\|_{\HH}^2=\sum_{n=0}^{K}|a_n|^2,
\qquad
E=\mathfrak a_0[f,f]
=\sum_{n=0}^{K}\lambda_n|a_n|^2,
\]
where \(\lambda_n=2n(2n+1)\). For any integer \(N\ge1\), split
\(f=P_{\le N}f+P_{>N}f\). By Cauchy--Schwarz and
\eqref{eq:endpoint-growth},
\[
|P_{\le N}f(0)|^2+|P_{\le N}f(\pi)|^2
\le
C(N+1)^{3/2}L.
\]
For the high modes,
\[
\begin{aligned}
|P_{>N}f(0)|^2+|P_{>N}f(\pi)|^2
&\le
E\sum_{n>N}
\frac{|\Phi_n(0)|^2+|\Phi_n(\pi)|^2}{\lambda_n}\\
&\le
CN^{-1/2}E,
\end{aligned}
\]
because \(\lambda_n\asymp n^2\) and the summand is
\(O(n^{-3/2})\). Hence
\begin{equation}\label{eq:trace-split}
|f(0)|^2+|f(\pi)|^2
\le
C\left((N+1)^{3/2}L+N^{-1/2}E\right).
\end{equation}
If \(E\le L\), choose \(N=1\). If \(E>L\), choose an integer
\(N\asymp(E/L)^{1/2}\). In both cases, \eqref{eq:trace-split}
yields
\[
|f(0)|^2+|f(\pi)|^2
\le
C\left(E^{3/4}L^{1/4}+L\right),
\]
which is \eqref{eq:trace}. The finite expansions are dense in
\(\mathcal V\), and \eqref{eq:trace-epsilon} shows that both endpoint
evaluations extend continuously to \(\mathcal V\). Finally,
\eqref{eq:trace-epsilon} follows from \eqref{eq:trace} and Young's
inequality
\[
E^{3/4}L^{1/4}
\le
\frac34\varepsilon E+\frac14\varepsilon^{-3}L.
\]
\end{proof}

\begin{remark}
The exponent \(3/4\) is the critical quantity. It generates the fourth
power in the lower Robin spectral scale and, after taking the radial square
root, the frequency scale \(\mu^2\).
\end{remark}

\section{Logarithmic cylinder and endpoint feedback}
\label{sec:cylinder}

Let
\[
w(t,\varphi)
=
W\!\left(\frac{e^{2t}}2\cos\varphi,
          e^t\sqrt{\sin\varphi}\right),
\qquad t=\log\rho.
\]
Proposition~\ref{prop:polar} gives, in the interior angular interval,
\begin{equation}\label{eq:cylinder-interior}
w_{tt}+w_t-A_0w=0.
\end{equation}
We next identify the endpoint feedback. For a smooth profile the chain rule
gives
\[
w_\varphi
=-\frac{\rho^2}{2}\sin\varphi\,W_x
+\frac{\rho\cos\varphi}{2\sqrt{\sin\varphi}}\,W_z,
\]
and hence the formal endpoint fluxes are
\begin{equation}\label{eq:angular-flux-endpoints}
\omega w_\varphi\big|_{\varphi=0}
=\frac{e^t}{2}\,\partial_zW\!\left(\frac{e^{2t}}2,0\right),
\qquad
\omega w_\varphi\big|_{\varphi=\pi}
=-\frac{e^t}{2}\,\partial_zW\!\left(-\frac{e^{2t}}2,0\right).
\end{equation}
These expressions only identify the coefficient and are not used as
pointwise traces for rough solutions.

Let \(C:\mathcal V\to\mathbb C^2\) be the continuous endpoint trace
\[
Cf=\begin{pmatrix}f(0)\\ f(\pi)\end{pmatrix}.
\]
If \(b=\kappa_{1/4}V\), define
\begin{equation}\label{eq:beta-exact}
\beta_0(t)=2e^t b\!\left(\frac{e^{2t}}2\right),
\qquad
\beta_\pi(t)=2e^t b\!\left(-\frac{e^{2t}}2\right),
\end{equation}
and
\[
\Gamma(t)=\operatorname{diag}(\beta_0(t),\beta_\pi(t)).
\]
For a finite-energy profile, apply the weak conormal identity
\eqref{eq:weak-robin-reduction} to tests supported in a compact Grushin
annulus and transform to \((t,\varphi)\). For every
\(\psi\in C_c^\infty(J;\mathcal V)\) this gives
\begin{equation}\label{eq:weak-cylinder-local}
\int_J\left(
-(w_t,\psi_t)_{\HH}+(w_t,\psi)_{\HH}
-\mathfrak a_0[w,\psi]
-(\Gamma(t)Cw,C\psi)_{\mathbb C^2}
\right)\,dt=0.
\end{equation}
The signs agree with the smooth calculation in
\eqref{eq:angular-flux-endpoints}: the upper angular endpoint has the
opposite geometric orientation, and the lower endpoint is subtracted in the
integration-by-parts bracket. Thus the endpoint condition is represented,
without a pointwise normal derivative, by the weak cylinder equation
\begin{equation}\label{eq:full-cylinder-equation}
w_{tt}+w_t-A_0w=C^*\Gamma(t)Cw
\qquad\text{in }\mathcal V^*.
\end{equation}
In particular, on any translated interval \(t=t_0+\tau\),
\begin{equation}\label{eq:beta-scale}
\|\Gamma\|_{L^\infty(J)}
\le
C e^{t_0+\sup J}\|V\|_{L^\infty(\mathcal A_{t_0,J})},
\end{equation}
where \(\mathcal A_{t_0,J}\) is the corresponding pair of boundary annuli.
Thus the dimensionless Robin size at Grushin radius \(\rho_0=e^{t_0}\) is
of order \(\rho_0\|V\|_\infty\), and its square is the critical cost scale.

\section{The Robin angular form and the critical scale}

For \(\boldsymbol\beta=(\beta_0,\beta_\pi)\in\R^2\), define on
\(\mathcal V\)
\[
\mathfrak a_{\boldsymbol\beta}[f,g]
=
\mathfrak a_0[f,g]
+
\beta_0f(0)\overline{g(0)}
+
\beta_\pi f(\pi)\overline{g(\pi)}.
\]
By Theorem~\ref{thm:trace}, this is a closed, lower-semibounded form and
defines a self-adjoint operator \(A_{\boldsymbol\beta}\) with compact
resolvent.

\begin{theorem}[Critical form lower bound and sharp power]
\label{thm:form-bound}
There exists \(C>0\) such that
\begin{equation}\label{eq:robin-lower}
\mathfrak a_{\boldsymbol\beta}[f,f]
\ge
-C\left(
1+|\beta_0^-|^4+|\beta_\pi^-|^4
\right)\|f\|_{\HH}^2
\end{equation}
for every \(f\in\mathcal V\). Moreover, there exist \(c>0\) and
\(\mu_0>0\) such that
\begin{equation}\label{eq:sharp-mu4}
\inf\sigma(A_{(-\mu,0)})\le-c\mu^4,
\qquad \mu\ge\mu_0.
\end{equation}
Thus the fourth power in \eqref{eq:robin-lower} is sharp.
\end{theorem}

\begin{proof}
Let
\[
\mu=\beta_0^-+\beta_\pi^-.
\]
Using \eqref{eq:trace-epsilon},
\[
\mathfrak a_{\boldsymbol\beta}[f,f]
\ge
E-C\mu\left(\varepsilon E+\varepsilon^{-3}L\right),
\]
where \(E=\mathfrak a_0[f,f]\) and \(L=\|f\|_{\HH}^2\). If
\(\mu\le1\), choose \(\varepsilon\) to be a fixed sufficiently small
constant. If \(\mu>1\), choose \(\varepsilon=(2C\mu)^{-1}\). This gives
\[
\mathfrak a_{\boldsymbol\beta}[f,f]
\ge
\frac12E-C(1+\mu^4)L.
\]
Since
\(\mu^4\le8(|\beta_0^-|^4+|\beta_\pi^-|^4)\),
\eqref{eq:robin-lower} follows.

To prove sharpness, choose a nonnegative function
\(g\in C_c^\infty([0,\infty))\) such that \(g(0)=1\) and
\[
4\int_0^\infty r^{1/2}|g'(r)|^2\dd r<1.
\]
Such a function is obtained by taking a cutoff that remains equal to one on
\([0,1]\) and decays sufficiently slowly on a long fixed interval. Define
\[
f_\mu(\varphi)=g(\mu^2\varphi).
\]
For large \(\mu\), its support lies near \(\varphi=0\). With
\(r=\mu^2\varphi\) and
\((\sin(r/\mu^2))^{1/2}=\mu^{-1}r^{1/2}(1+o(1))\), one has
\[
\mathfrak a_0[f_\mu,f_\mu]
=
4\mu\int_0^\infty r^{1/2}|g'(r)|^2\dd r+o(\mu),
\]
\[
\|f_\mu\|_{\HH}^2
=
\mu^{-3}\int_0^\infty r^{1/2}|g(r)|^2\dd r+o(\mu^{-3}),
\]
and \(f_\mu(0)=1\). Therefore
\[
\mathfrak a_{(-\mu,0)}[f_\mu,f_\mu]
\le-c_1\mu
\]
for all sufficiently large \(\mu\), whereas
\(\|f_\mu\|_{\HH}^2\le C_1\mu^{-3}\). The Rayleigh quotient then gives
\eqref{eq:sharp-mu4}.
\end{proof}

The radial rate associated with a spectral bottom of order \(-\mu^4\) is of
order \(\mu^2\). This motivates the critical cutoff
\[
N_{\mathrm{crit}}\asymp1+\mu^2.
\]

\section{Low--high trace splitting and high-frequency coercivity}

Let \(P_{\le N}\) and \(P_{>N}\) denote the spectral projectors of
\(A_0\).

\begin{lemma}[Low-frequency trace]\label{lem:low}
For every integer \(N\ge0\),
\[
|P_{\le N}f(0)|^2+|P_{\le N}f(\pi)|^2
\le
C(N+1)^{3/2}\|P_{\le N}f\|_{\HH}^2.
\]
\end{lemma}

\begin{proof}
Write \(P_{\le N}f=\sum_{n=0}^{N}a_n\Phi_n\). By
Cauchy--Schwarz and \eqref{eq:endpoint-growth},
\[
|P_{\le N}f(0)|^2
\le
\left(\sum_{n=0}^{N}|a_n|^2\right)
\left(\sum_{n=0}^{N}|\Phi_n(0)|^2\right)
\le
C(N+1)^{3/2}\|P_{\le N}f\|_{\HH}^2.
\]
The estimate at \(\pi\) is identical.
\end{proof}

\begin{lemma}[High-frequency trace]\label{lem:high}
For every integer \(N\ge1\),
\[
|P_{>N}f(0)|^2+|P_{>N}f(\pi)|^2
\le
CN^{-1/2}
\mathfrak a_0[P_{>N}f,P_{>N}f].
\]
\end{lemma}

\begin{proof}
For \(P_{>N}f=\sum_{n>N}a_n\Phi_n\), Cauchy--Schwarz gives
\[
|P_{>N}f(0)|^2
\le
\left(\sum_{n>N}\lambda_n|a_n|^2\right)
\left(\sum_{n>N}\frac{|\Phi_n(0)|^2}{\lambda_n}\right).
\]
Since \(|\Phi_n(0)|^2\le C(n+1)^{1/2}\) and
\(\lambda_n\asymp n^2\), the second factor is bounded by
\(CN^{-1/2}\). The endpoint \(\pi\) is treated in the same way.
\end{proof}

\begin{theorem}[High-frequency coercivity]\label{thm:high-coercivity}
There exists a constant \(C_0\) such that, if
\[
N\ge C_0(1+\mu^2),
\qquad
|\beta_0|+|\beta_\pi|\le\mu,
\]
then
\begin{equation}\label{eq:high-coercivity}
\mathfrak a_{\boldsymbol\beta}[f,f]
\ge
\frac12\mathfrak a_0[P_{>N}f,P_{>N}f]
-CN^2\|P_{\le N}f\|_{\HH}^2.
\end{equation}
\end{theorem}

\begin{proof}
Set \(f_L=P_{\le N}f\) and \(f_H=P_{>N}f\). Since
\(|a+b|^2\le2|a|^2+2|b|^2\), Lemmas~\ref{lem:low} and~\ref{lem:high} imply
\[
\begin{aligned}
\mathfrak a_{\boldsymbol\beta}[f,f]
&\ge
\mathfrak a_0[f_H,f_H]
-2\mu\sum_{\xi\in\{0,\pi\}}
\left(|f_L(\xi)|^2+|f_H(\xi)|^2\right)\\
&\ge
\left(1-C\mu N^{-1/2}\right)
\mathfrak a_0[f_H,f_H]
-C\mu(N+1)^{3/2}\|f_L\|_{\HH}^2.
\end{aligned}
\]
Choosing \(C_0\) sufficiently large makes the first coefficient at least
\(1/2\). The same assumption gives \(\mu\le C\sqrt N\), and hence
\[
\mu(N+1)^{3/2}\le CN^2.
\]
This proves \eqref{eq:high-coercivity}.
\end{proof}

\section{Free finite-dimensional observability}

We now prove the finite-dimensional boundary observability estimate used in
our spectral decomposition. The coefficient-recovery mechanism belongs to
the Tur'an--Nazarov and moment-method family for exponential systems
\cite{Nazarov,FR71,FR74}. The coefficient argument below is given explicitly because both the regular
and singular radial branches must be recovered on the same observation
interval.

\begin{lemma}[Polynomial coefficient inequality]\label{lem:polynomial-coeff}
There exists an absolute constant $C_*>1$ such that, for every integer
$m\ge0$ and every polynomial
\[
P(y)=\sum_{j=0}^{m}c_jy^j,
\]
one has
\begin{equation}\label{eq:polynomial-coeff}
\sum_{j=0}^{m}|c_j|^2
\le
C_*^{\,m+1}
\int_1^e|P(y)|^2\,\dd y.
\end{equation}
\end{lemma}

\begin{proof}
Let
\[
a=\frac{e+1}{2},
\qquad
b=\frac{e-1}{2},
\qquad
Q(x)=P(a+bx),
\quad -1\le x\le1.
\]
Write
\[
Q(x)=\sum_{k=0}^{m}\gamma_k L_k(x),
\]
where $\{L_k\}_{k\ge0}$ is the orthonormal Legendre basis of
$L^2(-1,1)$.  Parseval's identity gives
\begin{equation}\label{eq:legendre-parseval}
\sum_{k=0}^{m}|\gamma_k|^2
=
\|Q\|_{L^2(-1,1)}^2.
\end{equation}
The classical explicit formula
\[
P_k(x)
=
2^{-k}
\sum_{r=0}^{\lfloor k/2\rfloor}
(-1)^r
\binom{k}{r}
\binom{2k-2r}{k}
 x^{k-2r}
\]
for the Legendre polynomial $P_k$, together with
$L_k=\sqrt{(2k+1)/2}\,P_k$, implies that the Euclidean norm of the
monomial coefficient vector of $L_k$ is bounded by $C_0^{k+1}$ for an
absolute constant $C_0>1$.  Indeed,
\[
2^{-k}\binom{k}{r}\binom{2k-2r}{k}
\le 2^{-k}2^k2^{2k-2r}
\le 4^k,
\]
and the polynomial factors in $k$ can be absorbed into a larger exponential
constant.

If $Q(x)=\sum_{j=0}^{m}q_jx^j$, the triangle inequality followed by
Cauchy--Schwarz and \eqref{eq:legendre-parseval} yields
\[
\left(\sum_{j=0}^{m}|q_j|^2\right)^{1/2}
\le
\sum_{k=0}^{m}|\gamma_k|C_0^{k+1}
\le
C_1^{m+1}\|Q\|_{L^2(-1,1)}.
\]
Finally,
\[
P(y)=Q\!\left(\frac{y-a}{b}\right).
\]
The coefficient vector of $((y-a)/b)^j$ has Euclidean norm bounded by
\[
b^{-j}(1+a)^j.
\]
A second application of Cauchy--Schwarz therefore gives
\[
\left(\sum_{j=0}^{m}|c_j|^2\right)^{1/2}
\le
C_2^{m+1}
\left(\sum_{j=0}^{m}|q_j|^2\right)^{1/2}.
\]
Since
\[
\|Q\|_{L^2(-1,1)}^2
=
\frac1b\int_1^e|P(y)|^2\,\dd y,
\]
combining the preceding estimates and enlarging the constant proves
\eqref{eq:polynomial-coeff}.
\end{proof}

Let
\begin{equation}\label{eq:free-expansion}
z(t,\varphi)
=
\sum_{n=0}^{N}
\left(
 a_ne^{2nt}+d_ne^{-(2n+1)t}
\right)\Phi_n(\varphi),
\qquad 0\le t\le1.
\end{equation}
This is precisely the general solution in the truncated angular space of
\[
z_{tt}+z_t-A_0z=0.
\]
We define the Cauchy-data energy at $t=0$ by
\begin{equation}\label{eq:EN-definition}
\mathcal E_N[z](0)
:=
\|z_t(0)\|_{\HH}^2
+
\|(T+I)z(0)\|_{\HH}^2.
\end{equation}
The shift by \(I\) is essential only at the zero angular mode: unlike
\(A_0^{1/2}\), the operator \(T+I\) controls the constant component and
therefore defines a genuine Cauchy-data norm on the whole truncated space.

\begin{lemma}[Equivalence with radial coefficients]\label{lem:energy-equivalence}
There exist absolute constants $c,C>0$ such that
\begin{equation}\label{eq:energy-equivalence}
c\sum_{n=0}^{N}(n+1)^2
\bigl(|a_n|^2+|d_n|^2\bigr)
\le
\mathcal E_N[z](0)
\le
C\sum_{n=0}^{N}(n+1)^2
\bigl(|a_n|^2+|d_n|^2\bigr).
\end{equation}
\end{lemma}

\begin{proof}
For a fixed $n$, put
\[
x_n=a_n+d_n,
\qquad
y_n=2na_n-(2n+1)d_n.
\]
The contribution of this mode to \eqref{eq:EN-definition} is
\[
|y_n|^2+(2n+1)^2|x_n|^2.
\]
The upper estimate in \eqref{eq:energy-equivalence} follows from the definitions.
Conversely,
\[
a_n=\frac{y_n+(2n+1)x_n}{4n+1},
\qquad
d_n=\frac{2nx_n-y_n}{4n+1}.
\]
Since $(n+1)/(4n+1)$ is bounded above and below by positive absolute
constants, these identities give the lower estimate after summation in $n$.
\end{proof}

\begin{theorem}[Free boundary observability]\label{thm:free-obs}
There exists an absolute constant $C>0$ such that every function of the form
\eqref{eq:free-expansion} satisfies
\begin{equation}\label{eq:free-observability}
\mathcal E_N[z](0)
\le
e^{C(N+1)}
\int_0^1
\left(
 |z(t,0)|^2+|z(t,\pi)|^2
\right)\dd t.
\end{equation}
\end{theorem}

\begin{proof}
Set $p_n=\Phi_n(0)$.  Since
\[
\Phi_n(\pi)=(-1)^np_n,
\]
define the even and odd observations
\[
g_0(t)=\frac{z(t,0)+z(t,\pi)}2,
\qquad
g_1(t)=\frac{z(t,0)-z(t,\pi)}2.
\]
Then, for $\sigma\in\{0,1\}$,
\begin{equation}\label{eq:parity-observation}
g_\sigma(t)
=
\sum_{\substack{0\le n\le N\\ n\equiv\sigma\ ({\rm mod}\ 2)}}
 p_n\left(a_ne^{2nt}+d_ne^{-(2n+1)t}\right).
\end{equation}
Moreover, the parallelogram identity gives
\begin{equation}\label{eq:parallelogram-observation}
|g_0(t)|^2+|g_1(t)|^2
=
\frac12\left(|z(t,0)|^2+|z(t,\pi)|^2\right).
\end{equation}

Put $y=e^t\in[1,e]$ and, for each parity $\sigma$, define
\[
P_\sigma(y)
=
y^{2N+1}g_\sigma(\log y).
\]
By \eqref{eq:parity-observation},
\begin{equation}\label{eq:parity-polynomial}
P_\sigma(y)
=
\sum_{\substack{0\le n\le N\\ n\equiv\sigma\ ({\rm mod}\ 2)}}
 p_n\left(
 a_n y^{2N+1+2n}
 +d_n y^{2N-2n}
 \right).
\end{equation}
This is an algebraic polynomial of degree at most $4N+1$.  The exponents
$2N+1+2n$ are odd and lie above $2N$, whereas the exponents $2N-2n$ are even
and lie at most at $2N$.  Consequently, all coefficients displayed in
\eqref{eq:parity-polynomial} occupy distinct monomial positions.  Applying
Lemma~\ref{lem:polynomial-coeff} with $m=4N+1$ gives
\begin{equation}\label{eq:coefficient-control-parity}
\sum_{\substack{0\le n\le N\\ n\equiv\sigma\ ({\rm mod}\ 2)}}
 |p_n|^2\left(|a_n|^2+|d_n|^2\right)
\le
C^{N+1}\int_1^e|P_\sigma(y)|^2\,\dd y.
\end{equation}
On the interval $[1,e]$,
\[
|P_\sigma(y)|^2
=y^{4N+2}|g_\sigma(\log y)|^2.
\]
Since $\dd y=y\,\dd t$,
\begin{equation}\label{eq:polynomial-to-time}
\int_1^e|P_\sigma(y)|^2\,\dd y
=
\int_0^1e^{(4N+3)t}|g_\sigma(t)|^2\,\dd t
\le
e^{4N+3}\int_0^1|g_\sigma(t)|^2\,\dd t.
\end{equation}
The exact endpoint formula in Proposition~\ref{prop:endpoint-growth} implies
that $|p_n|^2\ge c_0>0$ for every $n\ge0$.  Therefore, after summing
\eqref{eq:coefficient-control-parity} over both parities and using
\eqref{eq:polynomial-to-time}--\eqref{eq:parallelogram-observation},
\begin{equation}\label{eq:radial-coefficient-observation}
\sum_{n=0}^{N}
\left(|a_n|^2+|d_n|^2\right)
\le
e^{C(N+1)}
\int_0^1
\left(|z(t,0)|^2+|z(t,\pi)|^2\right)\dd t.
\end{equation}
Multiplying the left-hand side by $(N+1)^2$ changes only the exponential
constant.  Hence
\[
\sum_{n=0}^{N}(n+1)^2
\left(|a_n|^2+|d_n|^2\right)
\le
e^{C(N+1)}
\int_0^1
\left(|z(t,0)|^2+|z(t,\pi)|^2\right)\dd t.
\]
The conclusion follows from Lemma~\ref{lem:energy-equivalence}.
\end{proof}

\begin{remark}[Observability cost]
The factor \(e^{CN}\) is the natural finite-dimensional scale: the observed
traces contain \(O(N)\) distinct exponential frequencies, so coefficient
recovery on a fixed interval cannot be uniform in \(N\). Only the upper bound
needed below is used.
\end{remark}

\section{Robust observability in the truncated angular space}
\label{sec:truncated-robust}

We next prove stability of the finite-dimensional observability estimate
under a measurable, time-dependent Robin feedback, without differentiating
the coefficient.

For an integer $N\ge0$, set
\[
\HH_N=\operatorname{span}\{\Phi_0,\ldots,\Phi_N\},
\]
and define the endpoint observation operator
\[
C_N:\HH_N\longrightarrow\mathbb C^2,
\qquad
C_Nf=\begin{pmatrix}f(0)\\ f(\pi)\end{pmatrix}.
\]
Its adjoint is understood with respect to the scalar product of $\HH_N$ and
the Euclidean scalar product of $\mathbb C^2$.  Lemma~\ref{lem:low} gives
\begin{equation}\label{eq:CN-norm}
\|C_N\|^2_{\HH_N\to\mathbb C^2}
\le C(N+1)^{3/2}.
\end{equation}

Let $\Gamma\in L^\infty((0,1);\mathbb C^{2\times2})$ and assume
\begin{equation}\label{eq:Gamma-bound}
\|\Gamma(t)\|_{\mathbb C^{2\times2}}\le\mu
\quad\text{for almost every }t\in(0,1).
\end{equation}
The diagonal choice
$\Gamma(t)=\operatorname{diag}(\beta_0(t),\beta_\pi(t))$
corresponds to the Robin problem.

\begin{lemma}[Finite-dimensional Duhamel trace bound]
\label{lem:duhamel-trace}
Let $F\in L^2((0,1);\HH_N)$ and let $d$ solve
\begin{equation}\label{eq:forced-zero-data}
\begin{cases}
 d_{tt}+d_t-A_0d=F,&0<t<1,\\
 d(0)=d_t(0)=0.
\end{cases}
\end{equation}
Then
\begin{equation}\label{eq:duhamel-trace}
\|C_Nd\|_{L^2(0,1;\mathbb C^2)}
\le
C e^{2N}(N+1)^{3/4}
\|F\|_{L^2(0,1;\HH_N)}.
\end{equation}
If $F=C_N^*\Gamma y$ with
$y\in L^2((0,1);\mathbb C^2)$, then
\begin{equation}\label{eq:duhamel-boundary-input}
\|C_Nd\|_{L^2(0,1;\mathbb C^2)}
\le
C\mu e^{2N}(N+1)^{3/2}
\|y\|_{L^2(0,1;\mathbb C^2)}.
\end{equation}
\end{lemma}

\begin{proof}
For $0\le n\le N$, define
\[
s_n(r)=\frac{e^{2nr}-e^{-(2n+1)r}}{4n+1},
\qquad r\ge0.
\]
The operator family
\[
S_N(r)\Phi_n=s_n(r)\Phi_n
\]
is the zero-displacement solution operator for the scalar equations
$c_n''+c_n'-2n(2n+1)c_n=f_n$.  Therefore
\[
d(t)=\int_0^tS_N(t-s)F(s)\,\dd s.
\]
For $0\le r\le1$,
\[
|s_n(r)|\le 2e^{2Nr},
\]
and consequently
\begin{equation}\label{eq:SN-bound}
\|S_N(r)\|_{\HH_N\to\HH_N}\le2e^{2N}.
\end{equation}
Using \eqref{eq:CN-norm}, \eqref{eq:SN-bound}, and Minkowski's
integral inequality (equivalently, Young's inequality for the Volterra
kernel \(\mathbf 1_{0<s<t}\)), the map
$f\mapsto\int_0^t f(s)\,\dd s$ is bounded on $L^2(0,1)$, and we obtain
\[
\begin{aligned}
\|C_Nd\|_{L^2}
&\le
\sup_{0\le r\le1}\|C_NS_N(r)\|
\left\|\int_0^t\|F(s)\|_{\HH_N}\,\dd s\right\|_{L^2_t}\\
&\le
Ce^{2N}(N+1)^{3/4}\|F\|_{L^2(0,1;\HH_N)},
\end{aligned}
\]
which proves \eqref{eq:duhamel-trace}.  If
$F=C_N^*\Gamma y$, then
\[
\|F(t)\|_{\HH_N}
\le
\|C_N\|\,\|\Gamma(t)\|\,|y(t)|
\le
C\mu(N+1)^{3/4}|y(t)|.
\]
Substitution into \eqref{eq:duhamel-trace} gives
\eqref{eq:duhamel-boundary-input}.
\end{proof}

\begin{theorem}[Robust truncated boundary observability]
\label{thm:robust-truncated}
Let $q\in H^2((0,1);\HH_N)$ solve
\begin{equation}\label{eq:truncated-feedback}
q_{tt}+q_t-A_0q=C_N^*\Gamma(t)C_Nq
\quad\text{in }(0,1),
\end{equation}
where \eqref{eq:Gamma-bound} holds.  There exists an absolute constant
$C_0\ge1$ such that, whenever
\begin{equation}\label{eq:N-mu-condition}
N+1\ge C_0(1+\mu^2),
\end{equation}
one has
\begin{equation}\label{eq:robust-truncated-observability}
\mathcal E_N[q](0)
\le
\exp\bigl(C(N+1)\bigr)
\int_0^1|C_Nq(t)|^2\,\dd t.
\end{equation}
The constant is independent of the measurable matrix function $\Gamma$.
\end{theorem}

\begin{proof}
Let $z$ be the free solution with the same Cauchy data as $q$ at $t=0$:
\[
z_{tt}+z_t-A_0z=0,
\qquad
z(0)=q(0),\quad z_t(0)=q_t(0).
\]
Put $d=q-z$.  Then $d$ has zero Cauchy data and satisfies
\[
d_{tt}+d_t-A_0d=C_N^*\Gamma(t)C_Nq.
\]
By Lemma~\ref{lem:duhamel-trace},
\begin{equation}\label{eq:z-trace-vs-q}
\begin{aligned}
\|C_Nz\|_{L^2}
&\le
\|C_Nq\|_{L^2}+\|C_Nd\|_{L^2}\\
&\le
\left[1+C\mu e^{2N}(N+1)^{3/2}\right]
\|C_Nq\|_{L^2}.
\end{aligned}
\end{equation}
The free observability theorem and the equality of the Cauchy data give
\[
\mathcal E_N[q](0)
=
\mathcal E_N[z](0)
\le
e^{C_1(N+1)}\|C_Nz\|_{L^2}^2.
\]
Under \eqref{eq:N-mu-condition}, one has
$\mu\le C\sqrt{N+1}$, and hence
\[
1+C\mu e^{2N}(N+1)^{3/2}
\le
\exp\bigl(C_2(N+1)\bigr).
\]
Combining this with \eqref{eq:z-trace-vs-q} and enlarging the exponential
constant proves \eqref{eq:robust-truncated-observability}.
\end{proof}

\begin{corollary}[Truncated system with boundary input]
\label{cor:truncated-input}
Let $q\in H^2((0,1);\HH_N)$ satisfy
\[
q_{tt}+q_t-A_0q=C_N^*\Gamma(t)y(t),
\qquad
y\in L^2((0,1);\mathbb C^2).
\]
If \eqref{eq:N-mu-condition} holds, then
\begin{equation}\label{eq:truncated-input-observability}
\mathcal E_N[q](0)
\le
\exp\bigl(C(N+1)\bigr)
\left(
\|C_Nq\|_{L^2(0,1)}^2+
\|y\|_{L^2(0,1)}^2
\right).
\end{equation}
\end{corollary}

\begin{proof}
Repeat the preceding proof and use
\eqref{eq:duhamel-boundary-input}.  The polynomial factors involving
$\mu$ and $N$ are absorbed into $\exp(C(N+1))$ under
\eqref{eq:N-mu-condition}.
\end{proof}

\begin{remark}[Scope of the truncated estimate]
Theorem~\ref{thm:robust-truncated} concerns the closed finite-dimensional
Robin system with measurable coefficients. Applied to \(q=P_{\le N}w\) in
the full problem, Corollary~\ref{cor:truncated-input} controls the
low-frequency Cauchy data by its endpoint trace and by the full endpoint
trace. A separate high-frequency estimate is required to control possible
cancellation between \(CP_{\le N}w\) and \(CP_{>N}w\), where
\(Cf=(f(0),f(\pi))^{\mathsf T}\).
\end{remark}

\section{Rough-Robin observability on an interior cylinder}
\label{sec:full-observability}

We now remove the low--high cancellation left unresolved by the truncated
argument. The essential point is to observe on an interval strictly inside
a larger cylinder.

For \(d>1\), set
\[
J_d=(-d,1+d),
\qquad I=(0,1),
\]
and define
\[
\mathcal M_d(w)=\sup_{t\in J_d}\|w(t)\|_{\HH}.
\]
We consider weak finite-energy solutions
\[
w\in L^2(J_d;\mathcal V)\cap H^1(J_d;\HH),
\qquad
w_{tt}\in L^2(J_d;\mathcal V^*),
\]
of
\begin{equation}\label{eq:abstract-full-feedback}
w_{tt}+w_t-A_0w=C^*\Gamma(t)Cw,
\qquad
\|\Gamma\|_{L^\infty(J_d)}\le\mu,
\end{equation}
in the sense of Definition~\ref{def:weak-cylinder}. Since
\(H^1(J_d;\HH)\hookrightarrow C(\overline{J_d};\HH)\), the outer norm
\(\mathcal M_d(w)\) is well-defined. The spectral Galerkin justification
of all mode expansions and limiting arguments is given in
Appendix~\ref{app:functional}. The logarithmic variable is used only in the
elliptic cylinder formulation.

For later use, write
\[
p_n=\Phi_n(0),
\qquad
v_n=C\Phi_n=p_n\begin{pmatrix}1\\(-1)^n\end{pmatrix}.
\]
By Proposition~\ref{prop:endpoint-growth},
\begin{equation}\label{eq:vn-growth}
|v_n|^2\asymp(n+1)^{1/2}.
\end{equation}

\begin{lemma}[Scalar Green kernel with an interior margin]
\label{lem:scalar-green-margin}
Let \(\lambda_n=2n(2n+1)\), \(n\ge1\), and let
\(c_n''+c_n'-\lambda_nc_n=f_n\) on \(J_d=(a,b)\). There is a
decomposition
\[
c_n=h_n+\mathcal G_nf_n,
\]
where \(h_n\) solves the homogeneous equation with the same values as
\(c_n\) at \(a,b\), while \(\mathcal G_nf_n\) has zero endpoint values.
For \(t\in I=(0,1)\),
\begin{equation}\label{eq:homogeneous-margin}
|h_n(t)|
\le
C_d e^{-cnd}
\bigl(|c_n(a)|+|c_n(b)|\bigr).
\end{equation}
The Green kernel \(G_n(t,s)\) of \(\mathcal G_n\) satisfies
\begin{equation}\label{eq:green-pointwise}
|G_n(t,s)|
\le
\frac{C}{n+1}e^{-c(n+1)|t-s|},
\end{equation}
and
\begin{equation}\label{eq:green-l2}
\sup_{t\in J_d}
\int_{J_d}|G_n(t,s)|^2\,\dd s
\le
C(n+1)^{-3}.
\end{equation}
\end{lemma}

\begin{proof}
Put \(c_n(t)=e^{-t/2}r_n(t)\) and
\(\kappa_n=2n+\tfrac12\). Then
\[
r_n''-\kappa_n^2r_n=e^{t/2}f_n.
\]
The homogeneous interpolation formula on \((a,b)\) is
\[
r_n^h(t)
=
\frac{\sinh(\kappa_n(b-t))}{\sinh(\kappa_n(b-a))}r_n(a)
+
\frac{\sinh(\kappa_n(t-a))}{\sinh(\kappa_n(b-a))}r_n(b).
\]
Since \(\operatorname{dist}(I,\{a,b\})=d\), this gives
\eqref{eq:homogeneous-margin}. The Dirichlet Green kernel for
\(\partial_t^2-\kappa_n^2\) is
\[
\widetilde G_n(t,s)
=-\frac{
\sinh(\kappa_n(t_<-a))
\sinh(\kappa_n(b-t_>))
}{
\kappa_n\sinh(\kappa_n(b-a))
},
\]
where \(t_<=\min\{t,s\}\) and \(t_>=\max\{t,s\}\). Hence
\[
|\widetilde G_n(t,s)|
\le
\frac{C}{\kappa_n}e^{-\kappa_n|t-s|}.
\]
Returning to \(c_n\) multiplies the kernel by \(e^{(s-t)/2}\), which is
absorbed by the exponential because \(\kappa_n-1/2=2n\). This proves
\eqref{eq:green-pointwise}; integration yields
\eqref{eq:green-l2}.
\end{proof}

\begin{lemma}[High-frequency endpoint decoupling]
\label{lem:high-trace-decoupling}
There exist constants \(C,c>0\) such that, for every integer \(N\ge1\),
\begin{equation}\label{eq:high-trace-decoupling}
\|CP_{>N}w\|_{L^2(I)}
\le
C\mu N^{-1/2}\|Cw\|_{L^2(J_d)}
+
C_de^{-cNd}\mathcal M_d(w).
\end{equation}
Moreover,
\begin{equation}\label{eq:high-bulk-decoupling}
\|P_{>N}w(0)\|_{\HH}^2
\le
C\mu^2N^{-3/2}\|Cw\|_{L^2(J_d)}^2
+
C_de^{-cNd}\mathcal M_d(w)^2.
\end{equation}
\end{lemma}

\begin{proof}
Expand \(w=\sum_{n\ge0}c_n\Phi_n\). The mode equation is
\[
c_n''+c_n'-\lambda_nc_n
=
\langle\Gamma(t)Cw(t),v_n\rangle_{\mathbb C^2}.
\]
Apply Lemma~\ref{lem:scalar-green-margin} to every \(n>N\). The forced
high-frequency endpoint trace is an integral operator with matrix kernel
\[
K_N(t,s)
=
\sum_{n>N}G_n(t,s)v_nv_n^*\Gamma(s).
\]
Using \eqref{eq:vn-growth} and \eqref{eq:green-pointwise},
\[
\|K_N(t,s)\|
\le
C\mu
\sum_{n>N}(n+1)^{-1/2}
 e^{-c(n+1)|t-s|}.
\]
Consequently,
\[
\sup_t\int_{J_d}\|K_N(t,s)\|\,\dd s
+
\sup_s\int_I\|K_N(t,s)\|\,\dd t
\le
C\mu\sum_{n>N}(n+1)^{-3/2}
\le C\mu N^{-1/2}.
\]
Schur's test gives the first term in
\eqref{eq:high-trace-decoupling}. For the homogeneous part,
\eqref{eq:homogeneous-margin}, Cauchy--Schwarz, and
\eqref{eq:vn-growth} give
\[
\|C h_{>N}\|_{L^2(I)}
\le
C_de^{-cNd}
\bigl(\|w(-d)\|_{\HH}+\|w(1+d)\|_{\HH}\bigr),
\]
which proves \eqref{eq:high-trace-decoupling}.

At \(t=0\), the forced coefficient satisfies, by
\eqref{eq:green-l2},
\[
\begin{aligned}
|c_n^{\rm f}(0)|^2
&\le
\|G_n(0,\cdot)\|_{L^2(J_d)}^2
\|\langle\Gamma Cw,v_n\rangle\|_{L^2(J_d)}^2\\
&\le
C\mu^2(n+1)^{-3}|v_n|^2
\|Cw\|_{L^2(J_d)}^2.
\end{aligned}
\]
Summing and using \(|v_n|^2\le C(n+1)^{1/2}\) yields
\[
\sum_{n>N}|c_n^{\rm f}(0)|^2
\le
C\mu^2N^{-3/2}\|Cw\|_{L^2(J_d)}^2.
\]
The homogeneous contribution follows directly from
\eqref{eq:homogeneous-margin} after summing in \(n\). This proves
\eqref{eq:high-bulk-decoupling}.
\end{proof}

\begin{theorem}[Full rough-Robin interior-cylinder observability]
\label{thm:full-cylinder}
There exist an absolute margin \(d_*>1\) and constants
\(C_0,C,c>0\) such that the following holds. Let \(w\) solve
\eqref{eq:abstract-full-feedback} on \(J_{d_*}\), with
\(\|\Gamma\|_\infty\le\mu\). Then, for every integer
\begin{equation}\label{eq:full-N-condition}
N+1\ge C_0(1+\mu^2),
\end{equation}
one has
\begin{equation}\label{eq:full-cylinder-spectral}
\|w(0)\|_{\HH}^2
\le
\exp\bigl(C(N+1)\bigr)
\|Cw\|_{L^2(J_{d_*})}^2
+
Ce^{-cN}\mathcal M_{d_*}(w)^2.
\end{equation}
\end{theorem}

\begin{proof}
Put \(q=P_{\le N}w\). On \(I=(0,1)\),
\[
q_{tt}+q_t-A_0q=C_N^*\Gamma(t)Cw.
\]
Corollary~\ref{cor:truncated-input} gives
\begin{equation}\label{eq:low-before-decoupling}
\mathcal E_N[q](0)
\le
\exp(C_1(N+1))
\left(
\|C_Nq\|_{L^2(I)}^2+
\|Cw\|_{L^2(I)}^2
\right).
\end{equation}
Since \(C_Nq=Cw-CP_{>N}w\),
Lemma~\ref{lem:high-trace-decoupling} and
\eqref{eq:full-N-condition}, with \(C_0\) large enough, yield
\[
\|C_Nq\|_{L^2(I)}
\le
C\|Cw\|_{L^2(J_{d_*})}
+C_{d_*}e^{-cNd_*}\mathcal M_{d_*}(w).
\]
Substitution into \eqref{eq:low-before-decoupling} gives
\begin{equation}\label{eq:low-after-decoupling}
\mathcal E_N[q](0)
\le
\exp(C_2(N+1))\|Cw\|_{L^2(J_{d_*})}^2
+C\exp(C_2N-cNd_*)\mathcal M_{d_*}(w)^2.
\end{equation}
Choose the fixed margin \(d_*\) so large that
\(cd_*>C_2+2\). The second term in
\eqref{eq:low-after-decoupling} is then bounded by
\(Ce^{-c_1N}\mathcal M_{d_*}(w)^2\).

Finally,
\[
\|w(0)\|_{\HH}^2
\le
2\|q(0)\|_{\HH}^2
+2\|P_{>N}w(0)\|_{\HH}^2.
\]
The first term is controlled by \(\mathcal E_N[q](0)\), and the second by
\eqref{eq:high-bulk-decoupling}. Since
\(\mu^2N^{-3/2}\le C\) under \eqref{eq:full-N-condition}, all observation
terms are absorbed into \(\exp(C(N+1))\|Cw\|_{L^2(J_{d_*})}^2\).
Renaming the constants proves \eqref{eq:full-cylinder-spectral}.
\end{proof}

\begin{corollary}[Critical interpolation cost]
\label{cor:full-interpolation}
There exist \(C>0\) and \(\vartheta\in(0,1)\) such that every solution in
Theorem~\ref{thm:full-cylinder} satisfies
\begin{equation}\label{eq:full-interpolation}
\|w(0)\|_{\HH}^2
\le
C\exp\!\bigl(C(1+\mu^2)\bigr)
\mathcal M_{d_*}(w)^{2(1-\vartheta)}
\|Cw\|_{L^2(J_{d_*})}^{2\vartheta}.
\end{equation}
\end{corollary}

\begin{proof}
If \(\|w(0)\|_{\HH}=0\), there is nothing to prove. Otherwise choose
\[
N
=
\left\lceil
C_0(1+\mu^2)
+
\frac1c\log\!\left(
\frac{2C\mathcal M_{d_*}(w)^2}{\|w(0)\|_{\HH}^2}
\right)
\right\rceil.
\]
Since \(\mathcal M_{d_*}(w)\ge\|w(0)\|_{\HH}\), this is admissible and the
last term in \eqref{eq:full-cylinder-spectral} is at most
\(\frac12\|w(0)\|_{\HH}^2\). Rearranging gives
\[
\|w(0)\|_{\HH}^{2(1+\alpha)}
\le
C e^{C(1+\mu^2)}
\mathcal M_{d_*}(w)^{2\alpha}
\|Cw\|_{L^2(J_{d_*})}^2
\]
for an absolute \(\alpha>0\). Taking the power
\(1/(1+\alpha)\) and setting \(\vartheta=(1+\alpha)^{-1}\) proves
\eqref{eq:full-interpolation}.
\end{proof}

\begin{remark}[Why an outer norm is retained]
A frequency-independent full observability estimate cannot hold. Indeed,
for \(z_n(t,\varphi)=e^{-(2n+1)t}\Phi_n(\varphi)\), the initial energy is
\(2(2n+1)^2\), whereas
\[
\int_0^1|Cz_n(t)|^2\,\dd t
=2|p_n|^2\int_0^1e^{-2(2n+1)t}\,\dd t
\asymp n^{-1/2}.
\]
The quotient grows like \(n^{5/2}\); hence a spectral cutoff or an outer
norm is structural rather than technical.
\end{remark}

\section{Critical annular interpolation for the quarter-Laplacian extension}
\label{sec:quarter-corollary}

Let \(V\in L^\infty(\mathbb R)\) and let
\(u\in L^\infty(\mathbb R)\cap H^{1/4}_{\mathrm{loc}}(\mathbb R)\)
be a distributional solution of \eqref{eq:frac-sch}. Let \(W\) be its
Caffarelli--Silvestre extension after the change \(y=z^2/2\), and fix a
Grushin radius \(\rho_0>0\). Proposition~\ref{prop:extension-realization}
shows that the associated logarithmic profile belongs to the weak energy
class required in Theorem~\ref{thm:full-cylinder}. Define the angular mass
\[
\mathscr H_{\rho_0}(W)
=
\int_0^\pi
\left|
W\!\left(
\frac{\rho_0^2}{2}\cos\varphi,
\rho_0\sqrt{\sin\varphi}
\right)
\right|^2
(\sin\varphi)^{1/2}\,\dd\varphi.
\]
Let
\[
\mathscr M_{\rho_0}(W)
=
\sup_{\tau\in J_{d_*}}
\left\|
W\!\left(
\frac{\rho_0^2e^{2\tau}}2\cos\varphi,
\rho_0e^\tau\sqrt{\sin\varphi}
\right)
\right\|_{\HH}.
\]
The associated boundary annulus is
\[
\mathcal A_{\rho_0}
=
\left\{
 x\in\mathbb R:
 \frac{\rho_0^2}{2}e^{-2d_*}<|x|<
 \frac{\rho_0^2}{2}e^{2(1+d_*)}
\right\}.
\]

\begin{corollary}[Quarter-Laplacian annular interpolation]
\label{cor:quarter-annular}
There exist \(C>0\) and \(\vartheta\in(0,1)\) such that
\begin{equation}\label{eq:quarter-annular}
\mathscr H_{\rho_0}(W)
\le
C\exp\!\left[
C\left(1+\rho_0^2
\|V\|_{L^\infty(\mathcal A_{\rho_0})}^2\right)
\right]
\mathscr M_{\rho_0}(W)^{2(1-\vartheta)}
\left(
\int_{\mathcal A_{\rho_0}}
\frac{|u(x)|^2}{|x|}\,\dd x
\right)^{\vartheta}.
\end{equation}
\end{corollary}

\begin{proof}
Apply Corollary~\ref{cor:full-interpolation} to the translated cylinder
\(t=\log\rho_0+\tau\). From \eqref{eq:beta-exact},
\[
\mu
\le
C\rho_0\|V\|_{L^\infty(\mathcal A_{\rho_0})},
\]
where the fixed factor \(e^{1+d_*}\) is absorbed into \(C\). At the two
angular endpoints,
\[
w(\tau,0)=u\!\left(\frac{\rho_0^2e^{2\tau}}2\right),
\qquad
w(\tau,\pi)=u\!\left(-\frac{\rho_0^2e^{2\tau}}2\right).
\]
Since \(\dd\tau=\dd x/(2|x|)\),
\[
\|Cw\|_{L^2(J_{d_*})}^2
=
\frac12
\int_{\mathcal A_{\rho_0}}
\frac{|u(x)|^2}{|x|}\,\dd x.
\]
Substitution into \eqref{eq:full-interpolation} proves
\eqref{eq:quarter-annular}.
\end{proof}

\section{The endpoint boundary resolvent}
\label{sec:boundary-resolvent}

The half-density substitution removes the first logarithmic derivative and
reveals the arithmetic spectral structure needed for a critical Carleman
estimate. Put
\[
L_0=A_0+\frac14 I,
\qquad
\kappa_n=2n+\frac12,
\qquad
L_0\Phi_n=\kappa_n^2\Phi_n.
\]
If $w$ solves \eqref{eq:abstract-full-feedback} and
$r(t)=e^{t/2}w(t)$, then
\begin{equation}\label{eq:half-density-equation}
r_{tt}-L_0r=C^*\Gamma(t)Cr.
\end{equation}
Recall that
\[
v_n=C\Phi_n=p_n\binom{1}{(-1)^n},
\qquad
p_n^2\asymp (n+1)^{1/2}.
\]
The form scale has the spectral description
\[
\|f\|_{\mathcal V}^2\asymp
\sum_{n\ge0}\kappa_n^2|(f,\Phi_n)_{\HH}|^2,
\qquad
\|F\|_{\mathcal V^*}^2\asymp
\sum_{n\ge0}\kappa_n^{-2}
|\langle F,\Phi_n\rangle|^2.
\]
For \(z\notin\{\pm\kappa_n:n\ge0\}\), define the extrapolated resolvent
\(R(z):\mathcal V^*\to\mathcal V\) by
\begin{equation}\label{eq:extrapolated-resolvent}
R(z)F
=\sum_{n\ge0}
\frac{\langle F,\Phi_n\rangle_{\mathcal V^*,\mathcal V}}
{z^2-\kappa_n^2}\,\Phi_n.
\end{equation}
On \(\HH\) this agrees with the usual resolvent \((z^2-L_0)^{-1}\).
To verify the extrapolated mapping property, write
\(F_n=\langle F,\Phi_n\rangle_{\mathcal V^*,\mathcal V}\). Then
\[
\begin{aligned}
\|R(z)F\|_{\mathcal V}^2
&\asymp
\sum_{n\ge0}
\frac{\kappa_n^2|F_n|^2}{|z^2-\kappa_n^2|^2}\\
&\le
\left(\sup_{n\ge0}
\frac{\kappa_n^4}{|z^2-\kappa_n^2|^2}\right)
\sum_{n\ge0}\kappa_n^{-2}|F_n|^2
\le C_z\|F\|_{\mathcal V^*}^2.
\end{aligned}
\]
The supremum is finite for fixed \(z\) away from the poles and tends to one
as \(n\to\infty\). Coefficientwise multiplication shows that
\((z^2-L_0)R(z)=I\) on \(\mathcal V^*\) and
\(R(z)(z^2-L_0)=I\) on \(\mathcal V\). Since
\(C^*:\mathbb C^2\to\mathcal V^*\), it follows that
\begin{equation}\label{eq:resolvent-boundary-series}
R(z)C^*g
=\sum_{n\ge0}\frac{v_n^*g}{z^2-\kappa_n^2}\,\Phi_n
\quad\text{in }\mathcal V,
\end{equation}
and the bounded trace map defines
\[
K(z)g:=CR(z)C^*g
=\sum_{n\ge0}\frac{v_nv_n^*g}{z^2-\kappa_n^2}.
\]
The matrix series is absolutely convergent because, for large \(n\),
\[
\frac{|v_n|^2}{|z^2-\kappa_n^2|}
\le C_z\kappa_n^{-3/2},
\]
and \(\sum_n\kappa_n^{-3/2}<\infty\). Hence
\(R(z)C^*\in\mathcal L(\mathbb C^2,\mathcal V)\) and
\(K(z)\in\mathcal L(\mathbb C^2)\).

\begin{definition}[Admissible Carleman parameter]
Fix $0<\delta<1/4$. A number $\tau\ge2$ is called $\delta$-admissible if
\begin{equation}\label{eq:admissible-tau}
\operatorname{dist}\left(\tau,
 \left\{2n+\frac12:n\in\mathbb N_0\right\}\right)\ge\delta.
\end{equation}
Every interval of length $2$ contains a $\delta$-admissible number, after
fixing for instance $\delta=1/8$.
\end{definition}

\begin{lemma}[A discrete Cauchy-transform bound]
\label{lem:discrete-cauchy}
Let $h>0$, $a\in[0,h)$, and let
$\lambda_j=a+hj$. Suppose $\tau\ge4h$,
$\operatorname{dist}(\tau,a+h\mathbb Z)\ge\delta$, and
$z=\tau+i\xi$, $\xi\in\mathbb R$. Let $J$ consist of the indices for which
$\tau/2\le\lambda_j\le2\tau$. If a sequence $b_j$ satisfies
\[
|b_j|\le A\tau^{-1/2},
\qquad
|b_{j+1}-b_j|\le A\tau^{-3/2},
\]
then
\begin{equation}\label{eq:discrete-cauchy}
\left|\sum_{j\in J}\frac{b_j}{z-\lambda_j}\right|
\le C_{h,\delta}A\tau^{-1/2}.
\end{equation}
\end{lemma}

\begin{proof}
Choose $j_0$ so that $d=\tau-\lambda_{j_0}$ satisfies
$|d|\le h/2$. Admissibility gives $|d|\ge\delta$. If
$j=j_0+k$, the condition $\tau/2\le\lambda_j\le2\tau$ implies
\[
-c_1\tau/h\le k\le c_2\tau/h
\]
with positive constants $c_1,c_2$ depending only on $h$; in particular,
the positive and negative ranges have comparable lengths. We write
\begin{equation}\label{eq:cauchy-decomposition}
\sum_{j\in J}\frac{b_j}{z-\lambda_j}
=b_{j_0}\sum_{k=-J_-}^{J_+}\frac1{d+i\xi-hk}
+\sum_{k=-J_-}^{J_+}
 \frac{b_{j_0+k}-b_{j_0}}{d+i\xi-hk}.
\end{equation}

We first prove the uniform lattice estimate
\begin{equation}\label{eq:lattice-cauchy-uniform}
\left|\sum_{k=-J_-}^{J_+}\frac1{d+i\xi-hk}\right|
\le C_{h,\delta}.
\end{equation}
For the imaginary part, put \(\eta=|\xi|\). The term with \(k=0\) is
bounded by \((2|d|)^{-1}\le(2\delta)^{-1}\). For \(k\ne0\),
\(|d-hk|\ge h(|k|-1/2)\), and therefore
\[
\eta\sum_{k\ne0}\frac1{(d-hk)^2+\eta^2}
\le
2\eta\sum_{m=1}^{\infty}
\frac1{h^2(m-1/2)^2+\eta^2}
\le \frac{C}{h},
\]
where the last inequality follows by comparison with the corresponding
integral. When \(\xi=0\), the imaginary part vanishes. For the real part set
$f_\xi(x)=x/(x^2+\xi^2)$. Pairing $k$ with $-k$ and applying the mean value
theorem yields, for $k\ge1$,
\[
|f_\xi(d-hk)+f_\xi(d+hk)|
\le \frac{C_h|d|}{(hk)^2+\xi^2}.
\]
The paired series is summable uniformly in $\xi$. The unpaired indices form
at most one tail whose endpoints are comparable multiples of $\tau/h$.
If $|\xi|\le4\tau$, comparison with
$\int_A^B x/(x^2+\xi^2)\,dx$ gives the logarithm of a fixed ratio. If
$|\xi|>4\tau$, every term is $O(\tau/\xi^2)$ and there are $O(\tau)$
terms. This proves \eqref{eq:lattice-cauchy-uniform}.

For the second sum in \eqref{eq:cauchy-decomposition},
\[
|b_{j_0+k}-b_{j_0}|
\le A\tau^{-3/2}|k|.
\]
Moreover, for $k\ne0$,
$|d+i\xi-hk|\ge c_h(1+|k|)$, after adjusting the constant for the
finitely many $|k|\le2$. Hence each nonzero remainder term is bounded by
$CA\tau^{-3/2}$, and the $k=0$ term vanishes. Since the number of indices
is $O(\tau)$, the remainder is $O(A\tau^{-1/2})$. The first term is
bounded by \eqref{eq:lattice-cauchy-uniform} and
$|b_{j_0}|\le A\tau^{-1/2}$. This proves
\eqref{eq:discrete-cauchy}.
\end{proof}

\begin{lemma}[Critical endpoint resolvent bounds]
\label{lem:endpoint-resolvent}
Fix $\delta\in(0,1/4)$. There exists $C_\delta>0$ such that for every
$\delta$-admissible $\tau\ge2$, every $\xi\in\mathbb R$, and
$z=\tau+i\xi$, one has
\begin{align}
\|R(z)\|_{\HH\to\HH}&\le C_\delta\tau^{-1},
\label{eq:R-HH}\\
\|CR(z)\|_{\HH\to\mathbb C^2}
 +\|R(z)C^*\|_{\mathbb C^2\to\HH}
&\le C_\delta\tau^{-3/4},
\label{eq:CR}\\
\|K(z)\|_{\mathbb C^2\to\mathbb C^2}
&\le C_\delta\tau^{-1/2},
\label{eq:K-critical}\\
\|zR(z)\|_{\HH\to\HH}
 +\|L_0^{1/2}R(z)\|_{\HH\to\HH}
&\le C_\delta,
\label{eq:bulk-resolvent-first}\\
\|zR(z)C^*\|+\|L_0^{1/2}R(z)C^*\|
&\le C_\delta\tau^{1/4}.
\label{eq:boundary-resolvent-first}
\end{align}
\end{lemma}

\begin{proof}
Throughout the proof constants may depend on $\delta$. We first dispose of
the compact parameter range omitted by the lattice lemma. Suppose
\(2\le\tau<16\). The finitely many modes with \(\kappa_n\le32\) are
uniformly separated from the poles by admissibility. For \(\kappa_n>32\),
one has \(|\tau-\kappa_n|\ge\kappa_n/2\). Hence, uniformly in \(\xi\),
\begin{align*}
\sum_{n\ge0}
\frac{\kappa_n^{1/2}}
{(\tau+\kappa_n)\sqrt{(\tau-\kappa_n)^2+\xi^2}}
&\le C_\delta+C\sum_{\kappa_n>32}\kappa_n^{-3/2}
\le C_\delta,\\
\sum_{n\ge0}
\frac{\kappa_n^{1/2}}
{|z^2-\kappa_n^2|^2}
&\le C_\delta+C\sum_{\kappa_n>32}\kappa_n^{-7/2}
\le C_\delta,\\
\sum_{n\ge0}
\frac{(|z|^2+\kappa_n^2)\kappa_n^{1/2}}
{|z^2-\kappa_n^2|^2}
&\le
\sum_{n\ge0}
\frac{\kappa_n^{1/2}}
{(\tau-\kappa_n)^2+\xi^2}
\le C_\delta.
\end{align*}
The first series bounds \(K(z)\), the second bounds \(CR(z)\), and the
third bounds the two first-order boundary multipliers. The scalar resolvent
and first-order bulk multipliers are treated by the same finite-mode/tail
split; on the tail their multipliers are \(O(\kappa_n^{-2})\) and
\(O(\kappa_n^{-1})\), respectively. Thus all five estimates hold uniformly
for \(2\le\tau<16\). Since \(\tau\) is then comparable to one, these uniform
bounds have exactly the displayed powers of \(\tau\). We may therefore
assume \(\tau\ge16\) when Lemma~\ref{lem:discrete-cauchy} is invoked on the
parity lattices of spacing \(h=4\).

The factorization
\begin{equation}\label{eq:resolvent-factorization}
|z^2-\kappa_n^2|
=|z-\kappa_n|\,|z+\kappa_n|
\ge (\tau+\kappa_n)
\sqrt{(\tau-\kappa_n)^2+\xi^2}
\end{equation}
and admissibility imply
\[
\sup_n\frac1{|z^2-\kappa_n^2|}\le C\tau^{-1},
\quad
\sup_n\frac{|z|+\kappa_n}{|z^2-\kappa_n^2|}\le C.
\]
This proves \eqref{eq:R-HH} and
\eqref{eq:bulk-resolvent-first}.

Since $C\Phi_n=v_n$ and $|v_n|^2\le C\kappa_n^{1/2}$,
\begin{equation}\label{eq:CR-spectral-sum}
\|CR(z)\|^2
\le C\sum_{n\ge0}
\frac{\kappa_n^{1/2}}{|z^2-\kappa_n^2|^2}.
\end{equation}
On $\kappa_n\le\tau/2$, the sum is
$O(\tau^{-4}\sum_{\kappa_n\le\tau/2}\kappa_n^{1/2})
=O(\tau^{-5/2})$. On $\kappa_n\ge2\tau$, it is bounded by
$C\sum_{\kappa_n\ge2\tau}\kappa_n^{-7/2}=O(\tau^{-5/2})$.
In the middle range, \eqref{eq:resolvent-factorization} gives
\[
\frac{\kappa_n^{1/2}}{|z^2-\kappa_n^2|^2}
\le C\tau^{-3/2}
\frac1{(\tau-\kappa_n)^2+\xi^2}.
\]
The last lattice sum is uniformly bounded because
$\operatorname{dist}(\tau,\{\kappa_n\})\ge\delta$. Thus
\eqref{eq:CR-spectral-sum} is $O(\tau^{-3/2})$, proving
\eqref{eq:CR}; the adjoint estimate is identical.

We next justify \eqref{eq:boundary-resolvent-first}. Its squared left-hand
operator norms are controlled by
\[
C\sum_{n\ge0}
\frac{(|z|^2+\kappa_n^2)\kappa_n^{1/2}}
 {|z^2-\kappa_n^2|^2}.
\]
If $|\xi|\le4\tau$, the part $\kappa_n\le4\tau$ is bounded by
$C\tau^2$ times \eqref{eq:CR-spectral-sum}, hence by $C\tau^{1/2}$;
the tail $\kappa_n>4\tau$ is bounded by
$C\sum_{\kappa_n>4\tau}\kappa_n^{-3/2}=C\tau^{-1/2}$.
If $|\xi|>4\tau$, then $|z\pm\kappa_n|\ge|\xi|$. For
$\kappa_n\le2|\xi|$ the sum is at most
$C|\xi|^{-4}\sum_{\kappa_n\le2|\xi|}
(|\xi|^2+\kappa_n^2)\kappa_n^{1/2}=C|\xi|^{-1/2}$,
while for $\kappa_n>2|\xi|$ it is bounded by
$C\sum_{\kappa_n>2|\xi|}
(|\xi|^2+\kappa_n^2)\kappa_n^{-7/2}=C|\xi|^{-1/2}$.
Taking square roots proves \eqref{eq:boundary-resolvent-first}.

It remains to prove \eqref{eq:K-critical}. In the orthonormal basis
$2^{-1/2}(1,1)$, $2^{-1/2}(1,-1)$, the matrix $K(z)$ is diagonal, with
entries
\[
2\sum_{n\equiv\varepsilon\ ({\rm mod}\ 2)}
\frac{p_n^2}{z^2-\kappa_n^2},
\qquad \varepsilon=0,1.
\]
The parity lattices are
$\kappa_{2j}=4j+\frac12$ and
$\kappa_{2j+1}=4j+\frac52$. The ranges
$\kappa_n\le\tau/2$ and $\kappa_n\ge2\tau$ are bounded absolutely by
$C\tau^{-1/2}$. In the middle range write
\[
\frac{p_n^2}{z^2-\kappa_n^2}
=\frac{b_n(z)}{z-\kappa_n},
\qquad b_n(z)=\frac{p_n^2}{z+\kappa_n}.
\]
By Proposition~\ref{prop:endpoint-growth},
\[
p_n^2=c_*\left(n+\frac14\right)
\frac{\Gamma(n+\frac12)}{\Gamma(n+1)}.
\]
The exact ratio
\[
\frac{p_{n+2}^2}{p_n^2}
=\frac{n+\frac94}{n+\frac14}
 \frac{(n+\frac12)(n+\frac32)}{(n+1)(n+2)}
=1+O((n+1)^{-1})
\]
therefore gives
$|p_{n+2}^2-p_n^2|\le C(n+1)^{-1/2}$. Consequently, on
$\tau/2\le\kappa_n\le2\tau$,
\[
|b_n(z)|\le C\tau^{-1/2},
\qquad
|b_{n+2}(z)-b_n(z)|\le C\tau^{-3/2}.
\]
Lemma~\ref{lem:discrete-cauchy}, applied separately to the two parity
lattices, gives a $C\tau^{-1/2}$ bound for both diagonal entries. This is
\eqref{eq:K-critical}.
\end{proof}

\begin{lemma}[Fourier--resolvent representation in the energy class]
\label{lem:fourier-resolvent-representation}
Let $\tau$ be $\delta$-admissible. Suppose
\[
v\in L^2(\mathbb R;\mathcal V)\cap H^1(\mathbb R;\HH),
\qquad g\in L^2(\mathbb R;\mathbb C^2),
\qquad F\in L^2(\mathbb R;\HH),
\]
and
\begin{equation}\label{eq:abstract-resolvent-equation}
((\partial_t+\tau)^2-L_0)v=C^*g+F
\quad\text{in }\mathcal D'(\mathbb R;\mathcal V^*).
\end{equation}
Then, for almost every $\xi\in\mathbb R$, with $z=\tau+i\xi$,
\begin{align}
\widehat v(\xi)
&=R(z)C^*\widehat g(\xi)+R(z)\widehat F(\xi)
\quad\text{in }\HH,
\label{eq:fourier-resolvent-v}\\
\widehat{Cv}(\xi)
&=K(z)\widehat g(\xi)+CR(z)\widehat F(\xi)
\quad\text{in }\mathbb C^2.
\label{eq:fourier-resolvent-trace}
\end{align}
Moreover, $L_0^{1/2}\widehat v$ and
$z\widehat v$ are obtained from \eqref{eq:fourier-resolvent-v} by applying
the corresponding resolvent multipliers. All identities hold in the
associated $L^2_\xi$ spaces.
\end{lemma}

\begin{proof}
Take the Fourier transform in the logarithmic variable. On every bounded
frequency interval, \(\xi^2\widehat v\in L^2_{\rm loc}(\mathcal V^*)\),
while \(L_0\widehat v\in L^2(\mathcal V^*)\). Hence
\eqref{eq:abstract-resolvent-equation}, initially an identity in tempered
\(\mathcal V^*\)-valued distributions, becomes for almost every \(\xi\)
\[
\bigl(z^2-L_0\bigr)\widehat v(\xi)
=C^*\widehat g(\xi)+\widehat F(\xi),
\qquad z=\tau+i\xi,
\]
in \(\mathcal V^*\). Admissibility excludes the poles
\(z=\pm\kappa_n\), and the extrapolated resolvent
\eqref{eq:extrapolated-resolvent} is the inverse of
\(z^2-L_0:\mathcal V\to\mathcal V^*\). Applying this inverse gives
\eqref{eq:fourier-resolvent-v} directly in \(\mathcal V\), and therefore
also in \(\HH\).

The multiplier estimates in Lemma~\ref{lem:endpoint-resolvent} show that
both terms in \eqref{eq:fourier-resolvent-v} belong to the asserted
\(L^2_\xi\) spaces. Since the Fourier transform is unitary on
\(L^2(\mathbb R;\mathcal V)\) and
\(C:\mathcal V\to\mathbb C^2\) is bounded, one has
\(C\widehat v=\widehat{Cv}\) almost everywhere. Applying \(C\) to the full
resolvent identity gives \eqref{eq:fourier-resolvent-trace}; no estimate for
a truncated boundary matrix is used. The assertions for
\(z\widehat v\) and \(L_0^{1/2}\widehat v\) follow by applying the final two
multiplier bounds in Lemma~\ref{lem:endpoint-resolvent}.
\end{proof}

\section{A critical rough-Robin Carleman estimate}
\label{sec:critical-carleman}

\begin{theorem}[Linear-weight endpoint Carleman estimate]
\label{thm:critical-carleman}
Fix $\delta\in(0,1/4)$. There is a constant $C>0$ such that the following
holds. Let $\Gamma\in L^\infty(\mathbb R;\mathbb R^{2\times2})$ be symmetric
and $\|\Gamma\|_\infty\le\mu$. Let $v$ be compactly supported in $t$ and
belong to the finite-energy class. Assume that, in
$L^2(\mathbb R;\mathcal V^*)$,
\begin{equation}\label{eq:carleman-forced-equation}
\bigl((\partial_t+\tau)^2-L_0\bigr)v
=C^*\Gamma(t)Cv+F,
\qquad F\in L^2(\mathbb R;\HH).
\end{equation}
If $\tau$ is $\delta$-admissible and
\begin{equation}\label{eq:tau-absorption}
\tau\ge C(1+\mu^2),
\end{equation}
then
\begin{equation}\label{eq:critical-carleman}
\tau\|v\|_{L^2_t\HH}
+\|v_t\|_{L^2_t\HH}
+\|L_0^{1/2}v\|_{L^2_t\HH}
+\tau^{3/4}\|Cv\|_{L^2_t}
\le C\|F\|_{L^2_t\HH}.
\end{equation}
The equation is understood in \(L^2(\mathbb R;\mathcal V^*)\). The
Fourier--resolvent identities used below involve the full resolvent
\(R(z):\mathcal V^*\to\mathcal V\) and the full boundary matrix
\(K(z)=CR(z)C^*\).
\end{theorem}

\begin{proof}
Set
\[
g(t)=\Gamma(t)Cv(t).
\]
Since $C:\mathcal V\to\mathbb C^2$ is bounded and
$\Gamma\in L^\infty$, one has $g\in L^2(\mathbb R;\mathbb C^2)$.
Lemma~\ref{lem:fourier-resolvent-representation} applies to
\eqref{eq:carleman-forced-equation}. Its trace identity, together with
\eqref{eq:K-critical}, \eqref{eq:CR}, and Plancherel, gives
\[
\|Cv\|_{L^2}
\le C\mu\tau^{-1/2}\|Cv\|_{L^2}
 +C\tau^{-3/4}\|F\|_{L^2}.
\]
After enlarging the constant in \eqref{eq:tau-absorption}, the first term is
absorbed, and therefore
\begin{equation}\label{eq:trace-carleman-part}
\tau^{3/4}\|Cv\|_{L^2}\le C\|F\|_{L^2}.
\end{equation}

The bulk identity \eqref{eq:fourier-resolvent-v},
\eqref{eq:R-HH}, and \eqref{eq:CR} imply
\[
\|v\|_{L^2}
\le C\mu\tau^{-3/4}\|Cv\|_{L^2}
 +C\tau^{-1}\|F\|_{L^2}
\le C\tau^{-1}\|F\|_{L^2},
\]
where the last step uses \eqref{eq:trace-carleman-part} and
$\mu\tau^{-1/2}\le C$. Applying instead the multiplier bounds
\eqref{eq:bulk-resolvent-first} and
\eqref{eq:boundary-resolvent-first} gives
\[
\|(\partial_t+\tau)v\|_{L^2}
+\|L_0^{1/2}v\|_{L^2}
\le C\mu\tau^{1/4}\|Cv\|_{L^2}+C\|F\|_{L^2}
\le C\|F\|_{L^2}.
\]
Finally,
$v_t=(\partial_t+\tau)v-\tau v$. Combining the preceding estimates proves
\eqref{eq:critical-carleman}. Every use of the resolvent and of the endpoint
trace is justified by Lemma~\ref{lem:fourier-resolvent-representation}; no
additional strong-solution approximation is required.
\end{proof}

\begin{corollary}[Carleman estimate for the unconjugated equation]
\label{cor:unconjugated-carleman}
Let
\[
\mathcal L_\Gamma r=r_{tt}-L_0r-C^*\Gamma(t)Cr.
\]
For either choice of sign and every admissible
$\tau\ge C(1+\mu^2)$,
\begin{align}\label{eq:unconjugated-carleman}
&\tau\|e^{\pm\tau t}r\|_{L^2_t\HH}
+\|e^{\pm\tau t}r_t\|_{L^2_t\HH}
+\|e^{\pm\tau t}L_0^{1/2}r\|_{L^2_t\HH}
+\tau^{3/4}\|e^{\pm\tau t}Cr\|_{L^2_t}\notag\\
&\hspace{35mm}\le
C\|e^{\pm\tau t}\mathcal L_\Gamma r\|_{L^2_t\HH}
\end{align}
for compactly supported $r$.
\end{corollary}

\begin{proof}
Apply Theorem~\ref{thm:critical-carleman} to
$v=e^{\mp\tau t}r$. The proof for the opposite sign uses
$z=-\tau+i\xi$; all resolvent bounds are unchanged because they depend on
$z^2$ and on the distance of $|\Re z|$ from the spectral lattice.
\end{proof}

\begin{lemma}[Cylinder Caccioppoli estimate]
\label{lem:cylinder-caccioppoli}
Let $r$ solve \eqref{eq:half-density-equation} on an interval $J^+$ and let
$J\Subset J^+$ have positive distance from the boundary. If
$\|\Gamma\|_\infty\le\mu$, then
\begin{equation}\label{eq:cylinder-caccioppoli}
\|r_t\|_{L^2(J;\HH)}
+\|L_0^{1/2}r\|_{L^2(J;\HH)}
\le C_J(1+\mu^2)\|r\|_{L^2(J^+;\HH)}.
\end{equation}
\end{lemma}

\begin{proof}
Test the weak equation with $\eta^2r$, where $\eta$ equals one on $J$ and
is supported in $J^+$. The endpoint term is estimated using the
infinitesimal trace inequality
\[
|Cr|^2\le C\left(\varepsilon\mathfrak a_0[r]
 +\varepsilon^{-3}\|r\|_{\HH}^2\right).
\]
Choose $\varepsilon=c(1+\mu)^{-1}$ and absorb the angular-form term. The
remaining zeroth-order contribution is bounded by
$C(1+\mu^4)\|r\|^2$. Young's inequality for the cutoff term proves the
squared version of \eqref{eq:cylinder-caccioppoli}; taking square roots gives
the statement.
\end{proof}

\section{Quantitative centered bulk propagation}
\label{sec:centered-bulk}

For an interval $I\subset\mathbb R$, write
\[
\mathcal N_I(r)=\|r\|_{L^2(I;\HH)}.
\]
All intervals below have fixed lengths; changing their endpoints by fixed
amounts only changes absolute constants.

\begin{lemma}[Three-cylinder inequality]
\label{lem:three-cylinder}
There exist $\theta\in(0,1)$ and $C>0$ such that every solution of
\eqref{eq:half-density-equation} on $(-4,3)$, with
$\|\Gamma\|_\infty\le\mu$, satisfies
\begin{equation}\label{eq:three-cylinder}
\mathcal N_{(-1,0)}(r)
\le C e^{C(1+\mu^2)}
\mathcal N_{(-3,-2)}(r)^\theta
\mathcal N_{(-4,3)}(r)^{1-\theta}.
\end{equation}
The same estimate holds after translating the logarithmic interval and the
spatial center of the Grushin coordinates.
\end{lemma}

\begin{proof}
Fix \(0\le\eta\le1\) with
\[
\eta\in C_c^\infty((-3,2)),
\qquad
\eta=1\quad\text{on }[-2,1],
\]
and choose it so that
\[
\operatorname{supp}\eta'\subset
J_-\cup J_+,
\quad
J_-=\left(-\frac{11}{4},-\frac94\right),
\quad
J_+=\left(\frac54,\frac74\right).
\]
The slightly enlarged intervals required by Caccioppoli may be taken inside
\((-3,-2)\) and \((1,2)\), respectively. Apply
Corollary~\ref{cor:unconjugated-carleman} to \(\eta r\) with the weight
\(e^{-\tau t}\), which increases toward the left. Since
\[
\mathcal L_\Gamma(\eta r)=2\eta'r_t+\eta''r,
\]
Lemma~\ref{lem:cylinder-caccioppoli} and the minimum of the weight on
\((-1,0)\) give, for every admissible
\(\tau\ge\tau_0:=C_*(1+\mu^2)\),
\begin{equation}\label{eq:three-cylinder-preoptimized}
X_2\le C(1+\mu^2)
\left(e^{a\tau}X_1+e^{-b\tau}X_3\right),
\qquad
a=\frac{11}{4},\quad b=\frac54,
\end{equation}
where
\[
X_1=\mathcal N_{(-3,-2)}(r),\qquad
X_2=\mathcal N_{(-1,0)}(r),\qquad
X_3=\mathcal N_{(-4,3)}(r).
\]
Indeed, on \(J_-\) one has \(e^{-\tau t}\le e^{11\tau/4}\), whereas on
\(J_+\) one has \(e^{-\tau t}\le e^{-5\tau/4}\); on \((-1,0)\) the
weight is at least one. This records the two weight gaps and the absorption
of the outer commutators explicitly.

Since \(X_1\le X_3\), the case \(X_1=0\) follows by sending
\(\tau\to\infty\) through admissible values. If \(X_1>0\), set
\[
\tau_*=(a+b)^{-1}\log(X_3/X_1),
\qquad
\theta=\frac{b}{a+b}=\frac5{16}.
\]
When \(\tau_*\ge\tau_0\), choose an admissible
\(\tau\in[\tau_*,\tau_*+2]\). Then
\[
e^{a\tau}X_1+e^{-b\tau}X_3
\le C X_1^\theta X_3^{1-\theta}.
\]
When \(\tau_*<\tau_0\), the elementary inclusion bound \(X_2\le X_3\)
and \(X_3/X_1<e^{(a+b)\tau_0}\) give
\[
X_2\le e^{b\tau_0}X_1^\theta X_3^{1-\theta}.
\]
Finally, the polynomial factor in \(1+\mu^2\) is absorbed into
\(e^{C(1+\mu^2)}\). The cutoff, the intervals, and the two numerical weight
gaps are invariant under translation in \(t\); translating the boundary
center changes neither the logarithmic equation nor these constants.
\end{proof}

\begin{proposition}[Uniform annular doubling from one fixed annulus]
\label{prop:uniform-annular-doubling}
Let $r$ solve \eqref{eq:half-density-equation} on $(-\infty,3)$ and assume
$\|\Gamma\|_\infty\le\mu$. Put
\[
G=\mathcal N_{(-4,3)}(r),
\qquad
A=\mathcal N_{(-1,0)}(r)>0,
\qquad
\Lambda=\log\frac{eG}{A}.
\]
There exists $C>0$ such that, for every $T\le-4$,
\begin{equation}\label{eq:uniform-annular-doubling}
\mathcal N_{(T,T+1)}(r)
\le
\exp\bigl(C(1+\mu^2+\Lambda)\bigr)
\mathcal N_{(T-3,T)}(r).
\end{equation}
Consequently, for every $T\le-4$, the three-unit inward block satisfies
\begin{equation}\label{eq:inward-lower-propagation}
\mathcal N_{(T-3,T)}(r)
\ge A\exp\bigl[-C(1+\mu^2+\Lambda)(1+|T|)\bigr].
\end{equation}
\end{proposition}

\begin{proof}
For a fixed $T\le-4$, choose $\eta_T$ equal to zero on
$(-\infty,T-2]$, equal to one on $[T-1,1]$, and zero on $[2,\infty)$.
Apply \eqref{eq:unconjugated-carleman} with the weight $e^{-\tau t}$ to
$\eta_Tr$. The commutator is supported in $(T-2,T-1)$ and $(1,2)$.
By Lemma~\ref{lem:cylinder-caccioppoli},
\begin{align}\label{eq:doubling-cutoff-pre}
&\tau^2 e^{-2\tau(T+1)}
 \mathcal N_{(T,T+1)}(r)^2
 +\tau^2 A^2\notag\\
&\quad\le C(1+\mu^2)^2
\left[
 e^{-2\tau(T-2)}\mathcal N_{(T-3,T)}(r)^2
 +e^{-2\tau}G^2
\right].
\end{align}
Choose an admissible
\[
\tau\in\left[
 C_0(1+\mu^2+\Lambda),
 C_0(1+\mu^2+\Lambda)+2
\right]
\]
with $C_0$ sufficiently large. Because
$G/A=e^{\Lambda-1}$, the last term in
\eqref{eq:doubling-cutoff-pre} is at most
$\frac12\tau^2A^2$. Dropping the $A^2$ term and comparing the weights
$e^{-2\tau(T+1)}$ and $e^{-2\tau(T-2)}$ gives
\[
\mathcal N_{(T,T+1)}(r)
\le C(1+\mu^2)e^{3\tau}
\mathcal N_{(T-3,T)}(r),
\]
which is \eqref{eq:uniform-annular-doubling}. On the other hand, dropping
the first term on the left of \eqref{eq:doubling-cutoff-pre} gives directly
\[
A\le C(1+\mu^2)\tau^{-1}e^{-\tau(T-2)}
\mathcal N_{(T-3,T)}(r).
\]
Since $T\le-4$ and
$\tau\le C(1+\mu^2+\Lambda)$, this rearranges to
\eqref{eq:inward-lower-propagation}, after enlarging the absolute constant.
\end{proof}

For $x_c\in\mathbb R$ and $\rho>0$, define the upper
Grushin gauge ball
\begin{equation}\label{eq:gauge-ball-definition}
\mathcal G_\rho(x_c)
=\left\{(x,z)\in\mathbb R\times(0,\infty):
 z^4+4(x-x_c)^2<\rho^4\right\}.
\end{equation}

\begin{lemma}[A quantitative boundary seed]
\label{lem:quantitative-boundary-seed}
Let $v$ be real-valued, $\|v\|_{L^\infty(\mathbb R)}\le C_0$, and suppose that
for some $\alpha\in(0,1/2)$,
\[
[v]_{C^\alpha(x_s-1/2,x_s+1/2)}\le H,
\qquad |v(x_s)|\ge c_0>0.
\]
Let $W(x,z)=U(x,z^2/2)$ be the transformed Caffarelli--Silvestre
extension. There are absolute constants $c,C>0$ such that, with
\[
\ell=\min\left\{\frac18,\left(\frac{c_0}{8H}\right)^{1/\alpha}\right\},
\qquad
 y_0=c\ell\left(\frac{c_0}{1+C_0}\right)^2,
\qquad z_0=(2y_0)^{1/2},
\]
one has, after multiplying $v$ by $-1$ if necessary,
\begin{equation}\label{eq:poisson-seed-pointwise}
W(x,z)\ge \frac{c_0}{8}
\quad\text{whenever}\quad
|x-x_s|\le\frac\ell4,\quad 0<z\le z_0.
\end{equation}
Consequently, every fixed gauge ball $\mathcal G_{\rho_s}(x_s)$ containing
this rectangle satisfies
\begin{equation}\label{eq:poisson-seed-norm}
\|zW\|_{L^2(\mathcal G_{\rho_s}(x_s))}
\ge c\,c_0\,\ell^{1/2}z_0^{3/2}.
\end{equation}
\end{lemma}

\begin{proof}
Assume $v(x_s)\ge c_0$. The choice of $\ell$ gives
$v\ge7c_0/8$ on $(x_s-\ell,x_s+\ell)$. For $s=1/4$, the Poisson kernel is
\[
P_y(h)=c_*\frac{y^{1/2}}{(h^2+y^2)^{3/4}},
\qquad \int_{\mathbb R}P_y(h)\,dh=1,
\]
and therefore, uniformly for $|x-x_s|\le\ell/4$,
\[
\int_{|x-h-x_s|>\ell/2}P_y(h)\,dh
\le C\left(\frac y\ell\right)^{1/2}.
\]
Splitting $U(x,y)=\int P_y(h)v(x-h)\,dh$ into this local region and its
complement gives
\[
U(x,y)\ge \frac{7c_0}{8}(1-\varepsilon)-C_0\varepsilon,
\qquad
\varepsilon\le C(y/\ell)^{1/2}.
\]
The definition of $y_0$ makes the right-hand side at least $c_0/8$ for
$0<y\le y_0$, proving \eqref{eq:poisson-seed-pointwise} after
$y=z^2/2$. Finally, integration over
$|x-x_s|\le\ell/4$, $0<z<z_0$, gives
\[
\|zW\|_{L^2}^2
\ge c c_0^2\ell\int_0^{z_0}z^2\,dz
=c c_0^2\ell z_0^3,
\]
which proves \eqref{eq:poisson-seed-norm}.
\end{proof}

\begin{lemma}[Fixed-scale gauge-ball propagation]
\label{lem:gauge-ball-propagation}
Set
\[
\rho_1=\frac12,
\qquad
\rho_2=\frac34,
\qquad
\rho_3=1,
\qquad
h_*=\frac18.
\]
There is \(\theta\in(0,1)\) with the following property. For every boundary
center \(x_c\in\mathbb R\), solutions of the Grushin--Robin problem satisfy
\begin{equation}\label{eq:gauge-three-ball}
\|zW\|_{L^2(\mathcal G_{\rho_2}(x_c))}
\le
C e^{C(1+M^2)}
\|zW\|_{L^2(\mathcal G_{\rho_1}(x_c))}^{\theta}
\|zW\|_{L^2(\mathcal G_{\rho_3}(x_c))}^{1-\theta},
\end{equation}
provided \(\|b\|_\infty\le M\) on the boundary projection of the largest
ball. If \(|x_{j+1}-x_j|\le h_*\), then
\begin{equation}\label{eq:gauge-overlap-containment}
\mathcal G_{\rho_1}(x_j)\subset
\mathcal G_{\rho_2}(x_{j+1}).
\end{equation}
Consequently, along every chain of at most $L_*$ such centers, where $L_*$
is fixed, a lower bound $B_0$ for the first small-ball norm propagates to
\begin{equation}\label{eq:gauge-chain-lower}
\|zW\|_{L^2(\mathcal G_{\rho_1}(x_{L_*}))}
\ge
\exp\{-C_{L_*}(1+M^2+|\log B_0|+\log(1+U))\},
\end{equation}
where $U$ is a common upper bound for the largest-ball norms.
\end{lemma}

\begin{proof}
In polar variables centered at \(x_c\),
\[
 z^2\,dx\,dz=\frac12\rho^4(\sin\varphi)^{1/2}
 \,d\rho\,d\varphi,
\]
and hence, with \(t=\log\rho\),
\begin{equation}\label{eq:gauge-log-norm}
\|zW\|_{L^2(\mathcal G_\rho(x_c))}^2
=\frac12\int_{-\infty}^{\log\rho}
 e^{5t}\|w_{x_c}(t)\|_{\HH}^2\,dt.
\end{equation}
Choose a cutoff \(\eta\) supported in
\((\log(1/4),0)\), equal to one on
\([\log(1/2),\log(3/4)]\), with
\[
\operatorname{supp}\eta'
\subset (\log(1/4),\log(3/8))
\cup(\log(7/8),\log(15/16)).
\]
The Caccioppoli enlargement of the left transition remains below
\(\log(1/2)\), while that of the right transition remains between
\(\log(3/4)\) and \(0\). Apply the Carleman estimate with weight
\(e^{-\tau t}\) to \(\eta r_{x_c}\), where
\(r_{x_c}=e^{t/2}w_{x_c}\). Since \(\rho_3=1\),
\[
\|\Gamma_{x_c}\|_\infty
\le C\rho_3\|b\|_\infty\le CM.
\]
Let \(Y\) be the \(L^2_t\HH\) norm of \(r_{x_c}\) on
\((\log(1/2),\log(3/4))\). The two explicit weight gaps give
\begin{equation}\label{eq:gauge-annular-preoptimized}
Y\le C(1+M^2)
\left(e^{a\tau}B_1+e^{-b\tau}B_3\right),
\qquad
a=\log3,
\quad b=\log\frac76,
\end{equation}
for admissible \(\tau\ge C(1+M^2)\), where
\(B_j=\|zW\|_{L^2(\mathcal G_{\rho_j}(x_c))}\). Indeed, after division by
the minimum weight on the observation annulus, the left transition contributes
\((\rho_2/(1/4))^\tau=3^\tau\), whereas the right transition contributes
\(((7/8)/\rho_2)^{-\tau}=(7/6)^{-\tau}\).

Optimizing exactly as in Lemma~\ref{lem:three-cylinder}, including the case
where the formal optimizer is below the Carleman threshold, yields
\[
Y\le Ce^{C(1+M^2)}B_1^\theta B_3^{1-\theta},
\qquad
\theta=\frac{b}{a+b}\in(0,1).
\]
Moreover, \eqref{eq:gauge-log-norm} and
\(e^{5t}\|w\|_{\HH}^2=e^{4t}\|r\|_{\HH}^2\) show that
\[
B_2\le B_1+CY
\le Ce^{C(1+M^2)}B_1^\theta B_3^{1-\theta},
\]
because \(B_1\le B_3\). This proves \eqref{eq:gauge-three-ball} with the
annuli, cutoff supports, and weight gaps fixed numerically.

For the overlap geometry, the Euclidean triangle inequality in the
variables \((z^2,2x)\) shows more generally that
\[
\mathcal G_{\rho_1}(x_j)\subset\mathcal G_{\rho_2}(x_{j+1})
\quad\text{whenever}\quad
|x_{j+1}-x_j|<\frac{\rho_2^2-\rho_1^2}{2}.
\]
For the chosen values,
\[
h_*=\frac18<\frac{\rho_2^2-\rho_1^2}{2}=\frac5{32},
\]
and, if \((x,z)\in\mathcal G_{\rho_1}(x_j)\), then
\[
\begin{aligned}
z^4+4|x-x_{j+1}|^2
&\le\left(\rho_1^2+2h_*\right)^2\\
&=\left(\frac14+\frac14\right)^2
=\frac14<\frac{81}{256}=\rho_2^4.
\end{aligned}
\]
Thus the containment has a strict uniform margin.

Set \(B_j=\|zW\|_{L^2(\mathcal G_{\rho_1}(x_j))}\). The overlap makes the
middle norm at \(x_{j+1}\) at least \(B_j\), while the largest-ball norm is
at most a common upper bound \(U\). Rearranging the three-ball estimate gives
\[
B_{j+1}
\ge
\left(Ce^{C(1+M^2)}U^{1-\theta}\right)^{-1/\theta}
B_j^{1/\theta}.
\]
Writing \(Y_j=-\log B_j\) and iterating for \(j\le L_*\) gives
\[
Y_j\le C_{L_*}
\left(|\log B_0|+1+M^2+\log(1+U)\right),
\]
which is \eqref{eq:gauge-chain-lower}. Only boundary-centered balls are
used, so no translation in the degenerate variable is involved.
\end{proof}

\begin{lemma}[Fixed-scale propagation from a normalized boundary point]
\label{lem:fixed-scale-lower}
Let $v$ be a real bounded solution of
\[
(-\Delta)^{1/4}v+qv=0\quad\text{in }\mathbb R,
\qquad
\|q\|_\infty\le M,
\qquad
\|v\|_\infty\le C_0,
\]
and suppose that $|v(x_s)|\ge c_0>0$ for some $|x_s|\le2$. For the Grushin extension centered at the origin, the radii fixed in
Lemma~\ref{lem:gauge-ball-propagation} may be used so that the resulting
logarithmic annulus is $(-1,0)$ and
\begin{equation}\label{eq:fixed-scale-lower}
\mathcal N_{(-1,0)}(e^{t/2}w)
\ge
\exp\left[-C\left(1+M^2+
 |\log c_0|+\log(1+C_0)\right)\right].
\end{equation}
Moreover, the corresponding norm on $(-4,3)$ is bounded above by
$C(1+C_0)(1+M)^C$.
\end{lemma}

\begin{proof}
Interior Schauder estimates for $(-\Delta)^{1/4}$ with bounded right-hand
side \cite{SilvestreHolder} give, for any fixed $\alpha\in(0,1/2)$,
\[
[v]_{C^\alpha(x_s-1/2,x_s+1/2)}
\le C_\alpha(1+M)C_0.
\]
Lemma~\ref{lem:quantitative-boundary-seed}, with
\(H=C_\alpha(1+M)C_0\), gives a positive lower bound on the seed rectangle.
By taking the absolute constant in \(y_0\) small enough, one has
\(z_0\le1/4\). Since \(\ell\le1/8\), the full seed rectangle in
\eqref{eq:poisson-seed-pointwise} satisfies
\[
z^4+4|x-x_s|^2
\le \frac1{256}+4\left(\frac1{32}\right)^2
=\frac1{128}<\frac1{16}=\rho_1^4,
\]
and is therefore contained in \(\mathcal G_{1/2}(x_s)\). Hence its
weighted bulk norm \(B_0\) satisfies
\begin{equation}\label{eq:initial-seed-log}
|\log B_0|
\le C\left(1+\log(1+M)+|\log c_0|+\log(1+C_0)\right).
\end{equation}
The Poisson kernel is positive and has unit mass, so $|W|\le C_0$; hence
all fixed largest-ball norms in what follows are bounded by an absolute
multiple of $C_0$.

Choose the terminal center
\[
x_c=\frac14.
\]
For \((x,z)\in\mathcal G_{1/2}(1/4)\), the Euclidean triangle inequality in
the variables \((z^2,2x)\) gives
\[
\frac14
<\left(z^4+4x^2\right)^{1/2}
<\frac34.
\]
Thus the terminal small ball lies in the explicit origin-centered annulus
\[
\frac12<\left(z^4+4x^2\right)^{1/4}<\frac{\sqrt3}{2},
\]
which is contained in \(e^{-1}<\rho<1\). Join \(x_s\) to \(x_c\) by
\[
x_j=x_s+j\frac{x_c-x_s}{N},
\qquad
N=\max\left\{1,\left\lceil8|x_c-x_s|\right\rceil\right\}\le18.
\]
Then \(|x_{j+1}-x_j|\le1/8=h_*\), so the chain has the absolute length
\(L_*=18\). Applying
\eqref{eq:gauge-chain-lower} and using \eqref{eq:initial-seed-log} yields
\[
\|zW\|_{L^2(\mathcal G_{\rho_1}(x_c))}
\ge
\exp\left[-C\left(1+M^2+|\log c_0|+\log(1+C_0)\right)\right].
\]
On the terminal ball, the Grushin radius relative to the origin stays in a
fixed compact subinterval of $(e^{-1},1)$. Using the exact weighted Jacobian
in \eqref{eq:gauge-log-norm}, its $\|zW\|_{L^2}$ norm is therefore at most
a fixed multiple of $\mathcal N_{(-1,0)}(e^{t/2}w)$, proving
\eqref{eq:fixed-scale-lower}. Finally, $|W|\le C_0$ directly gives the
stated upper bound on the fixed cylinder (in fact no positive power of
$M$ is needed).
\end{proof}

\begin{theorem}[Centered quantitative vanishing estimate]
\label{thm:centered-vanishing-proved}
Let $q\in L^\infty(\mathbb R;\mathbb R)$ and
$v\in L^\infty(\mathbb R;\mathbb R)$ satisfy
\[
(-\Delta)^{1/4}v+qv=0\quad\text{in }\mathcal S'(\mathbb R),
\qquad
\|q\|_\infty\le M,
\qquad
\|v\|_\infty\le C_0.
\]
Assume that \(|v(x_s)|\ge c_0\in(0,1]\) at some \(|x_s|\le2\). Then
there exist \(r_0,c,C>0\), depending only on \(C_0\), such that
\begin{equation}\label{eq:centered-final}
\|v\|_{L^\infty(-r,r)}
\ge c\,r^{C(1+M^2+|\log c_0|)}
\qquad(0<r<r_0).
\end{equation}
If the pointwise normalization is replaced by
\(\|v\|_{L^2(-2,2)}\ge K\in(0,1]\), then
\begin{equation}\label{eq:centered-final-L2}
\|v\|_{L^\infty(-r,r)}
\ge c\,r^{C(1+M^2+|\log K|)},
\qquad 0<r<r_0,
\end{equation}
with the same dependence of the constants.
\end{theorem}

\begin{proof}
Center the Grushin coordinates at zero and put $r_c(t)=e^{t/2}w(t)$.
Lemma~\ref{lem:fixed-scale-lower} and the fixed-scale upper bound show that
in Proposition~\ref{prop:uniform-annular-doubling}
\[
\mu\le CM,
\qquad
\Lambda\le C\left(1+M^2+|\log c_0|+\log(1+C_0)\right).
\]
Put
\[
\mathcal A=1+M^2+|\log c_0|+\log(1+C_0).
\]
The fixed-scale bounds imply \(\Lambda\le C\mathcal A\). Therefore, for
every \(T\ll-1\) there is a unit subinterval
\(I_T\subset(T-3,T)\) such that
\begin{equation}\label{eq:angular-shell-lower}
\int_{I_T}\|w(t)\|_{\HH}^2\,dt
\ge
\exp\left[-C\mathcal A(1+|T|)\right].
\end{equation}
Here the factor \(e^{t/2}\) between \(r_c\) and \(w\) contributes only an
additional absolute multiple of \(|T|\), already absorbed by
\(\mathcal A(1+|T|)\).

Choose $T$ so negative that the boundary annulus associated with every
radius $\rho_0=e^{t_0}$, $t_0\in I_T$, is contained in $(-r,r)$; this means
$T=\frac12\log r-O(1)$. By averaging, $t_0\in I_T$ can be selected so that
$\mathscr H_{\rho_0}(W)$ is bounded below by the right-hand side of
\eqref{eq:angular-shell-lower}. The Poisson maximum principle gives
$\mathscr M_{\rho_0}(W)\le C_0$. Rearranging
Corollary~\ref{cor:quarter-annular}, and using
$\rho_0^2M^2\le C M^2$, yields
\[
\int_{\mathcal A_{\rho_0}}
\frac{|v(x)|^2}{|x|}\,dx
\ge
\exp\left[-C\mathcal A(1+|\log r|)\right].
\]
On this fixed-ratio annulus, $|x|\asymp r$. Since the annulus is contained
in $(-r,r)$,
\[
\int_{-r}^{r}|v(x)|^2\,dx
\ge
r\exp\left[-C\mathcal A(1+|\log r|)\right].
\]
The inequality
$\|v\|_{L^\infty(-r,r)}^2\ge(2r)^{-1}\|v\|_{L^2(-r,r)}^2$
now gives \eqref{eq:centered-final}, after decreasing \(r_0<1\) and
absorbing the factor \(\exp(-C\mathcal A)\) into a fixed enlargement of the
power. Since \(\log(1+C_0)\) is fixed, the displayed constants depend only
on \(C_0\).

If only \(\|v\|_{L^2(-2,2)}\ge K\) is known, the interior H\"older
estimate used in Lemma~\ref{lem:fixed-scale-lower} provides a continuous
representative of \(v\). This representative has a point
\(x_s\in(-2,2)\) with \(|v(x_s)|\ge K/2\); otherwise
the \(L^2\) norm on an interval of length four would be smaller than \(K\).
Applying the preceding argument with \(c_0=K/2\) gives
\eqref{eq:centered-final-L2}.
\end{proof}

\begin{proof}[Proof of Theorem~\ref{thm:quant-landis}]
The second assertion of
Theorem~\ref{thm:centered-vanishing-proved} is the centered endpoint
estimate in Definition~\ref{def:centered-vanishing}. The scaling reduction
in Theorem~\ref{thm:scaling-reduction} gives
\eqref{eq:global-landis-target}.
\end{proof}

\begin{remark}[Critical scale]
Since \(p_n^2\asymp\kappa_n^{1/2}\), the parity-lattice cancellation in
\eqref{eq:K-critical} leaves the factor \(\tau^{-1/2}\). Robin feedback is
therefore absorbed for \(\tau\gtrsim\mu^2\). Under the Landis scaling,
\(\mu\asymp R^{1/2}\), so evaluation at radius \(R^{-1}\) produces the rate
\(R\log R\).
\end{remark}

\section{Conclusion}

The critical quarter-Laplacian admits an exact Grushin--Robin formulation
whose angular spectrum is sufficiently rigid to replace differentiability
of the potential by arithmetic resolvent cancellation. The endpoint
boundary-resolvent estimate yields a rough-coefficient Carleman inequality
with the critical parameter threshold $\tau\simeq\|V\|_\infty^2$. Combining
its centered bulk doubling consequence with the rough-Robin annular
observability theorem transfers quantitative non-vanishing to the boundary.
The final scaling argument gives the critical lower bound
$\exp(-CR\log R)$.

\appendix
\section{Weak formulation, Galerkin passage, and realization by the extension}
\label{app:functional}

This appendix records the functional-analytic framework used in
Section~\ref{sec:full-observability} and verifies that the logarithmic profile
of the quarter-Laplacian extension belongs to that framework on every compact
Grushin annulus.

\begin{definition}[Weak cylinder solution]\label{def:weak-cylinder}
Let $J\subset\mathbb R$ be a bounded interval and
$\Gamma\in L^\infty(J;\mathbb C^{2\times2})$. A function $w$ is a weak
finite-energy solution of
\[
w_{tt}+w_t-A_0w=C^*\Gamma(t)Cw
\]
if
\[
w\in L^2(J;\mathcal V)\cap H^1(J;\HH),
\qquad
w_{tt}\in L^2(J;\mathcal V^*),
\]
and, for every $\psi\in C_c^\infty(J;\mathcal V)$,
\begin{equation}\label{eq:weak-cylinder-definition}
\int_J\!\left(
\langle w_{tt},\psi\rangle_{\mathcal V^*,\mathcal V}
 +(w_t,\psi)_{\HH}
 -\mathfrak a_0[w,\psi]
 -(\Gamma Cw,C\psi)_{\mathbb C^2}
\right)\dd t=0.
\end{equation}
Equivalently, after integration by parts in $t$, the first term can be
replaced by $-(w_t,\psi_t)_{\HH}$.
\end{definition}

\begin{proposition}[Spectral Galerkin passage]\label{prop:galerkin-passage}
Let $w$ be a weak solution in the sense of
Definition~\ref{def:weak-cylinder}, and write
$c_n(t)=(w(t),\Phi_n)_{\HH}$. Then $c_n\in H^2(J)$ and
\begin{equation}\label{eq:weak-mode-equation}
c_n''+c_n'-\lambda_nc_n
=\langle\Gamma(t)Cw(t),v_n\rangle_{\mathbb C^2}
\quad\text{a.e.\ on }J.
\end{equation}
Moreover, with $P_m$ denoting the angular spectral projector,
\[
P_mw\longrightarrow w
\quad\text{in }L^2(J;\mathcal V)\cap H^1(J;\HH),
\]
\[
CP_mw\longrightarrow Cw
\quad\text{in }L^2(J;\mathbb C^2),
\qquad
P_mw\longrightarrow w
\quad\text{in }C(\overline J;\HH).
\]
Consequently, the finite-mode Green representations and all estimates in
Section~\ref{sec:full-observability} pass to weak finite-energy solutions.
\end{proposition}

\begin{proof}
Testing \eqref{eq:weak-cylinder-definition} with
$\psi(t)=\eta(t)\Phi_n$, $\eta\in C_c^\infty(J)$, gives
\eqref{eq:weak-mode-equation} in distributions. The right-hand side belongs
to $L^2(J)$ because the trace map $C:\mathcal V\to\mathbb C^2$ is bounded;
therefore $c_n\in H^2(J)$.

The spectral representation of the closed form gives
\[
\|f\|_{\mathcal V}^2\asymp
\sum_{n\ge0}(1+\lambda_n)|(f,\Phi_n)_{\HH}|^2.
\]
Dominated convergence in the corresponding coefficient sums yields
$P_mw\to w$ in $L^2(J;\mathcal V)$ and
$P_mw_t\to w_t$ in $L^2(J;\HH)$. Trace continuity then gives
$CP_mw\to Cw$ in $L^2(J;\mathbb C^2)$. Finally,
$P_mw\to w$ in $H^1(J;\HH)$, and the one-dimensional embedding
$H^1(J;\HH)\hookrightarrow C(\overline J;\HH)$ gives uniform convergence in
$\HH$. Applying the finite-mode estimates first to
$P_mw$ (with the full trace $Cw$ as the projected boundary input), then
letting $m\to\infty$, is justified by these convergences, Fatou's lemma, and
the uniform Schur bounds of Lemma~\ref{lem:high-trace-decoupling}.
\end{proof}

\begin{lemma}[Local regularity and the cutoff commutator]
\label{lem:local-fractional-regularity}
Let $u\in L^\infty(\mathbb R)$ and $f\in L^2_{\rm loc}(\mathbb R)$ satisfy
\[
(-\Delta)^{1/4}u=f\quad\text{in }\mathcal S'(\mathbb R).
\]
Then $u\in H^{1/2}_{\rm loc}(\mathbb R)$. More precisely, for every
$\chi\in C_c^\infty(\mathbb R)$,
\begin{equation}\label{eq:local-regularity-estimate}
\|\chi u\|_{H^{1/2}(\mathbb R)}
\le C_\chi\left(\|\chi f\|_{L^2(\mathbb R)}
+\|u\|_{L^\infty(\mathbb R)}\right).
\end{equation}
\end{lemma}

\begin{proof}
For $P=(-\Delta)^{1/4}$, the singular-integral formula gives, first for
smooth bounded functions,
\begin{equation}\label{eq:fractional-commutator-formula}
[P,\chi]u(x)
=c\int_{\mathbb R}
\frac{\chi(x)-\chi(y)}{|x-y|^{3/2}}u(y)\,dy.
\end{equation}
For a general bounded distributional solution, apply this identity to
$u_\varepsilon=u*\varrho_\varepsilon$ and pass to the limit in
$\mathcal S'$; the bounds below are uniform in $\varepsilon$. The integral
in \eqref{eq:fractional-commutator-formula} is absolutely convergent: near $x=y$ the numerator is
$O(|x-y|)$, and away from the diagonal the kernel
$|x-y|^{-3/2}$ is integrable. On a fixed compact set it is bounded by
$C_\chi\|u\|_\infty$, while outside a larger compact set it decays like
$C_\chi\|u\|_\infty(1+|x|)^{-3/2}$. Hence
\[
\|[P,\chi]u\|_{L^2(\mathbb R)}
\le C_\chi\|u\|_{L^\infty(\mathbb R)}.
\]
Distributionally,
$P(\chi u)=\chi f+[P,\chi]u\in L^2(\mathbb R)$. Since
\[
\|g\|_{H^{1/2}}^2\asymp
\|g\|_{L^2}^2+\|Pg\|_{L^2}^2,
\]
we obtain \eqref{eq:local-regularity-estimate} and the asserted local
regularity.
\end{proof}

\begin{proposition}[Realization by the quarter-Laplacian extension]
\label{prop:extension-realization}
Let $V\in L^\infty_{\rm loc}(\mathbb R)$ and
$u\in L^\infty(\mathbb R)$ satisfy
\[
(-\Delta)^{1/4}u+Vu=0
\quad\text{in }\mathcal S'(\mathbb R).
\]
Let $U$ be the Caffarelli--Silvestre Poisson extension and
$W(x,z)=U(x,z^2/2)$. For every $0<\rho_-<\rho_+<\infty$, set
$J=(\log\rho_-,\log\rho_+)$ and
\[
w(t,\varphi)=W\!\left(\frac{e^{2t}}2\cos\varphi,
 e^t\sqrt{\sin\varphi}\right).
\]
Then $w$ is a weak finite-energy solution on $J$ in the sense of
Definition~\ref{def:weak-cylinder}, with the diagonal coefficient
\eqref{eq:beta-exact}. In particular,
\[
w\in L^2(J;\mathcal V)\cap H^1(J;\HH),
\qquad
w_{tt}\in L^2(J;\mathcal V^*).
\]
\end{proposition}

\begin{proof}
Here \(f=-Vu\in L^2_{\rm loc}\), and
Lemma~\ref{lem:local-fractional-regularity} gives
\(u\in H^{1/2}_{\rm loc}\). Let
\(Q=I\times(0,Y)\) be a compact half-cylinder containing the Grushin annulus
under consideration. Choose \(\chi_0\in C_c^\infty(\mathbb R)\) equal to
one on a neighborhood of \(I\), and write
\(u=u_1+u_2\), where \(u_1=\chi_0u\). Then
\(u_1\in H^{1/2}(\mathbb R)\subset H^{1/4}(\mathbb R)\), so its extension
has finite global weighted energy.

The tail \(u_2\) is supported a positive distance from \(I\). Direct
differentiation of the quarter-Poisson kernel gives, for \(x\in I\) and
\(0<y<Y\),
\[
|\partial_xU_2(x,y)|\le C y^{1/2}\|u\|_\infty,
\qquad
|\partial_yU_2(x,y)|\le C y^{-1/2}\|u\|_\infty.
\]
Hence \(y^{1/2}|\nabla U_2|^2\le
C(y^{3/2}+y^{-1/2})\), and therefore
\begin{equation}\label{eq:local-extension-energy}
\int_Q y^{1/2}|\nabla U|^2\,dx\,dy<\infty.
\end{equation}
Under \(y=z^2/2\), this is equivalent, up to the constant \(2^{-1/2}\), to
\begin{equation}\label{eq:local-grushin-energy}
\int\left(|W_z|^2+z^2|W_x|^2\right)\,dx\,dz<\infty
\end{equation}
on compact Grushin annuli.

We next establish the weak identity for bounded data. Let \(\chi_R\in C_c^\infty(\mathbb R)\) satisfy
\(\chi_R=1\) on \((-R,R)\), \(0\le\chi_R\le1\), and put
\(u_R=\chi_Ru\). For \(R\) large enough relative to \(I\),
\[
u_R=u_1+u_{2,R},
\qquad
u_{2,R}=(\chi_R-\chi_0)u,
\]
and Lemma~\ref{lem:local-fractional-regularity}, applied with the cutoff
\(\chi_R\), gives \(u_R\in H^{1/2}(\mathbb R)\subset H^{1/4}(\mathbb R)\).
Its extension \(U_R\) satisfies the
global energy identity
\begin{equation}\label{eq:truncated-extension-identity}
\int_{y>0}y^{1/2}\nabla U_R\cdot\nabla\Phi\,dx\,dy
=\frac1{d_{1/4}}
\langle(-\Delta)^{1/4}u_R,\Phi(\cdot,0)\rangle
\end{equation}
for every smooth compactly supported test \(\Phi\).

On \(Q\), the same kernel bounds as above hold for \(U_{2,R}\) with
constants independent of \(R\). Moreover,
\(\nabla U_{2,R}\to\nabla U_2\) pointwise. The integrable majorant
\(C(y^{3/2}+y^{-1/2})\) therefore yields convergence in the local weighted
energy pairing, while the \(u_1\) part is fixed. Thus the left-hand side of
\eqref{eq:truncated-extension-identity} converges to the corresponding
pairing with \(U\).

For \(\phi=\Phi(\cdot,0)\), distributional duality gives
\[
\langle(-\Delta)^{1/4}u_R,\phi\rangle
=\int_{\mathbb R}u_R(-\Delta)^{1/4}\phi\,dx.
\]
Since \((-\Delta)^{1/4}\phi\in L^1(\mathbb R)\), bounded convergence implies
that this tends to
\(\langle(-\Delta)^{1/4}u,\phi\rangle\). Passing to the limit in
\eqref{eq:truncated-extension-identity} and using
\((-\Delta)^{1/4}u=-Vu\) proves
\[
\int_{y>0}y^{1/2}\nabla U\cdot\nabla\Phi\,dx\,dy
=-\frac1{d_{1/4}}\int_{\mathbb R}V(x)u(x)\Phi(x,0)\,dx.
\]
Finally set \(\Phi(x,y)=\Psi(x,\sqrt{2y})\). Although this function need not
be smooth at \(y=0\), it belongs to the weighted energy test space. Indeed,
with \(z=\sqrt{2y}\),
\begin{align*}
\int_{y>0} y^{1/2}|\Phi_y|^2\,dx\,dy
&=\frac1{\sqrt2}\int_{z>0}|\Psi_z|^2\,dx\,dz,\\
\int_{y>0} y^{1/2}|\Phi_x|^2\,dx\,dy
&=\frac1{\sqrt2}\int_{z>0}z^2|\Psi_x|^2\,dx\,dz.
\end{align*}
Its trace is \(\Psi(\cdot,0)\), and approximation in the weighted energy
norm by smooth tests makes it admissible in the extension identity. The
change \(y=z^2/2\) multiplies the Grushin energy pairing by
\(2^{-1/2}\), and hence gives
\begin{equation}\label{eq:weak-grushin-robin}
\int_{z>0}\left(W_z\Psi_z+z^2W_x\Psi_x\right)\,dx\,dz
=-\int_{\mathbb R}b(x)u(x)\Psi(x,0)\,dx,
\end{equation}
where \(b=\kappa_{1/4}V\). This proves the conormal identity, including the contribution of the bounded
tail, for every test supported in a compact Grushin annulus.

The coordinate Jacobian and the bilinear Grushin energy identity are
\[
\left|\frac{\partial(x,z)}{\partial(\rho,\varphi)}\right|
=\frac{\rho^2}{2\sqrt{\sin\varphi}},
\]
\begin{equation}\label{eq:energy-coordinate-identity}
\int_{\rho_-<\rho<\rho_+}
\left(|W_z|^2+z^2|W_x|^2\right)\,dx\,dz
=\frac12\int_J e^t
\left(\|w_t\|_{\HH}^2+\mathfrak a_0[w,w]\right)\,dt.
\end{equation}
The Poisson kernel is positive and has unit mass, so
$|w(t,\varphi)|\le\|u\|_\infty$ and hence $w\in L^2(J;\HH)$. Since $e^t$
is bounded above and below on $J$,
\eqref{eq:energy-coordinate-identity} then gives
$w_t\in L^2(J;\HH)$ and $w\in L^2(J;\mathcal V)$. Thus
$w\in L^2(J;\mathcal V)\cap H^1(J;\HH)$. Moreover,
\eqref{eq:beta-exact} and $V\in L^\infty_{\rm loc}$ imply
$\Gamma\in L^\infty(J;\mathbb R^{2\times2})$.

It remains to derive the cylinder equation directly at the energy level. First
take $\psi\in C_c^\infty(J;C^\infty([0,\pi]))$ and use in
\eqref{eq:weak-grushin-robin} the test function whose polar representation
is
\[
\Psi(\rho,\varphi)=2\rho^{-1}\psi(\log\rho,\varphi).
\]
A direct change of variables gives
\begin{align}
&\int_{z>0}\left(W_z\Psi_z+z^2W_x\Psi_x\right)\,dx\,dz
\notag\\
&\qquad=\int_J\left(
(w_t,\psi_t)_{\HH}-(w_t,\psi)_{\HH}
+\mathfrak a_0[w,\psi]\right)\,dt.
\label{eq:weak-coordinate-transform}
\end{align}
On the boundary $z=0$, the two components are
$x=e^{2t}/2$ and $x=-e^{2t}/2$, while
$|dx|=e^{2t}dt$. Hence the right-hand side of
\eqref{eq:weak-grushin-robin} becomes
\[
-\int_J(\Gamma(t)Cw(t),C\psi(t))_{\mathbb C^2}\,dt,
\]
with $\Gamma$ given by \eqref{eq:beta-exact}. Integrating the first term in
\eqref{eq:weak-coordinate-transform} by parts in $t$ yields exactly
\eqref{eq:weak-cylinder-definition}. Density of
$C_c^\infty(J;C^\infty([0,\pi]))$ in
$C_c^\infty(J;\mathcal V)$ and the continuity of
$C:\mathcal V\to\mathbb C^2$ extend the identity to every admissible test
function.

Finally, the identity defines
\[
w_{tt}=-w_t+A_0w+C^*\Gamma Cw
\quad\text{in }\mathcal V^*.
\]
The three terms on the right belong to $L^2(J;\mathcal V^*)$: the first by
$\HH\hookrightarrow\mathcal V^*$, the second by the bounded form map
$\mathcal V\to\mathcal V^*$, and the third by the trace inequality and
$\Gamma\in L^\infty(J)$. Thus
$w_{tt}\in L^2(J;\mathcal V^*)$ and the proposition follows.
\end{proof}

\begin{remark}
The preceding proposition is local in the Grushin radius. On each compact
logarithmic interval it identifies the concrete extension with the weak
cylinder energy class used in the interpolation estimates; no evolution
interpretation of the logarithmic variable is involved.
\end{remark}

\end{document}